\documentclass[12pt]{article}

\usepackage{amssymb,amsmath,amsthm}
\usepackage[shortlabels]{enumitem}
\usepackage{tikz}
\usepackage{graphicx}
\usetikzlibrary{arrows.meta}
\usepackage[T1]{fontenc}

\newcommand{\A}{\mathcal{A}}
\newcommand{\D}{\mathcal{D}}
\newcommand{\G}{\mathcal{G}}
\newcommand{\Z}{\mathbb{Z}}

\newcommand{\N}{\mathcal{N}}
\newcommand{\C}{\mathcal{C}}
\newcommand{\QQ}{\mathcal{Q}}
\newcommand{\Q}{\mathbb{Q}}
\newcommand{\F}{\mathbb{F}}

\newcommand{\OO}{\mathcal{O}}
\newcommand{\M}{\mathcal{M}}

\newcommand{\Kb}{\overline{K}}

\newcommand{\ra}{\rightarrow}

\newcommand{\lra}{\longrightarrow}

\newcommand{\norm}{\mathrm{N}}

\newcommand{\etab}{\overline{\eta}}
\newcommand{\sigmab}{\overline{\sigma}}
\newcommand{\rhob}{\overline{\rho}}
\newcommand{\rhoh}{\widehat{\rho}}

\newcommand{\alphah}{\widehat{\alpha}}
\newcommand{\gammab}{\overline{\gamma}}

\DeclareMathOperator{\Gal}{Gal}

\DeclareMathOperator{\Aut}{Aut}

\newtheorem{theorem}{Theorem}
\newtheorem{lemma}[theorem]{Lemma}
\newtheorem{prop}[theorem]{Proposition}
\newtheorem{cor}[theorem]{Corollary}

\theoremstyle{definition}
\newtheorem{remark}[theorem]{Remark}
\newtheorem{definition}[theorem]{Definition}

\numberwithin{equation}{section}
\numberwithin{theorem}{section}

\title{Some nonabelian subgroups of the Nottingham group
over $\F_4$}
\author{Kevin Keating \\
Department of Mathematics \\
University of Florida \\
Gainesville, FL 32611 \\
USA \\[.2cm]
{\tt keating@ufl.edu}}

\begin{document}

\maketitle

\begin{abstract}
We classify the conjugacy classes of minimally ramified
nonabelian subgroups of order 8 in the Nottingham group
$\N(\F_4)$.  We then use finite automata to give
explicit descriptions of representatives for each of
these conjugacy classes.
\end{abstract}

\section{Introduction}

Let $k$ be a finite field of characteristic $p$ and set
$\N(k)={\{t+a_2t^2+a_3t^3+\dots:a_i\in k\}}$.  For
$\phi(t),\psi(t)\in\N(k)$ define 
$(\phi\circ\psi)(t)=\phi|_{t=\psi}$ to be the series
obtained from $\phi(t)$ by replacing $t$ with $\psi(t)$.
This substitution operation makes $\N(k)$ a group, known
as the Nottingham group of $k$.  It follows from a
theorem of Witt \cite{witt} that every finite $p$-group
is isomorphic to a subgroup of $\N(k)$.  Klopsch
\cite{klopsch} gave closed-form formulas for
representatives of every conjugacy class of elements of
order $p$ in $\N(k)$.  This leads to an explicit
classification of the conjugacy classes of cyclic
subgroups of $\N(k)$ with order $p$.  Finding formulas
for elements of more complicated finite subgroups of
$\N(k)$ has proven to be difficult.  For instance, all
the known examples of formulas for elements of $\N(k)$
of order $p^d$ with $d\ge2$ have $p=2$ and $d=2$.

     It was proposed in \cite{BCT} that elements of
finite order in $\N(k)$ should be described using finite
automata.  The basis for this approach is Christol's
theorem \cite{Chr}, which says that
$\sum_{n=0}^{\infty}a_nt^n\in k[[t]]$ is algebraic over
the rational function field $k(t)$ if and only if the
coefficient sequence $(a_n)_{n\ge0}$ is the output of a
finite automaton of an appropriate type.  To construct a
subgroup $\G$ of $\N(k)$ which is isomorphic to a given
$p$-group $G$, we start with a a totally ramified
$G$-extension $L/k((u))$.  Then $L\cong k((t))$, so for
each $\sigma\in\Gal(L/k((u)))$, $\sigma(t)$ is an
element of $\N(k)$.  In order to describe $\sigma(t)$
explicitly we choose $L$ to be a $G$-extension of
$k((u))$ which is generated by elements which are
algebraic over $k(u)$.  We then express $t$ and
$\sigma(t)$ as rational functions of these generators.
This allows us to compute a polynomial $f(t,X)$ with
coefficients in $k$ which has $X=\sigma(t)$ as a root.
Applying an explicit version of Christol's theorem gives
a finite automaton which generates the coefficients of
$\sigma(t)$.

     In \cite{BCT} the authors gave explicit
constructions of finite automata which determine
generators for some finite abelian subgroups of
$\N(\F_2)$.  In Section~8.3 of that paper they suggested
that it would be interesting to find automata for
generators of finite nonabelian subgroups of
$\N(\F_{2^m})$ as well.  In this paper we construct
finite automata which determine representatives for the
conjugacy classes of subgroups of $\N(\F_4)$ which are
isomorphic to the quaternion group $Q_8$ and have lower
ramification breaks $1,1,3$.  These are the smallest
possible lower breaks for a $Q_8$-subgroup of
$\N(\F_4)$.  We also construct finite automata which
correspond to representatives for the conjugacy classes
of subgroups of $\N(\F_4)$ which are isomorphic to the
dihedral group $D_4$ and have lower ramification breaks
$1,1,5$.  These are the smallest possible lower breaks
for a $D_4$-subgroup of $\N(\F_4)$.

\section{Finite automata}

Let $k\cong\F_q$ be a finite field.  We wish to describe
certain elements of $\N(k)$ using finite automata.  Our
automata have input alphabet $\{0,1,2,\dots,q-1\}$ and
output alphabet $k$.  An automaton is a directed
multigraph with loops such that every vertex has
outdegree $q$.  The edges emanating from each vertex are
labeled with the elements of the input alphabet
$\{0,1,2,\dots,q-1\}$, with each label appearing exactly
once.  The vertices are labeled with elements of the
output alphabet $k$, with the restriction that if an
edge with label 0 connects $v_1$ to $v_2$ then the
vertices $v_1,v_2$ must have the same label.  In
addition, one of the vertices is marked Start.  In what
follows, we often refer to the vertices of the automaton
as ``states''.

   An automaton can be used to generate an ``automatic
sequence'' as follows: Let $n\ge0$ and write $n$ in base
$q$ as $d_rd_{r-1}\dots d_1d_0$.  Beginning at the Start
vertex, trace the path though the digraph given by the
edges labeled $d_0,d_1,\dots,d_{r-1},d_r$.  The $n$th
term in our automatic sequence is the label attached to
the final vertex in this path.  Thanks to the
restriction on the vertex labeling, adding leading 0s to
our base-$q$ representation of $n$ doesn't change the
output.  Christol's theorem \cite{Chr} says that the
series $\sum_{n=0}^{\infty}a_nt^n\in k[[t]]$ is
algebraic over $k(t)$ if and only if $(a_n)_{n\ge0}$ is
the output of an automaton of the type described here.

\section{Ramification theory}

In this section we collect definitions and facts about
higher ramification theory.  See Chapter~IV of \cite{cl}
for more information.  As above, we assume that
$k\cong\F_q$ is a finite field of characteristic $p$.

     Let $\A(k)=\{a_1t+a_2t^2+\cdots\in k[[t]]:
a_1\not=0\}$.  Then $\A(k)$ with the operation of
substitution forms a group which fits into an exact
sequence
\[1\lra\N(k)\lra\A(k)\lra k^{\times}\lra1.\]
Let $E$ be a local field of characteristic $p$ with
residue field $k$ and let $\pi_E$ be a uniformizer for
$E$.  Let $\sigma\in\Aut_k(E)$.  Then there is a
uniquely determined $\phi_{\sigma}(t)\in\A(k)$ such that
$\sigma(\pi_E)=\phi_{\sigma}|_{t=\pi_E}$.  Since
$\sigma$ is necessarily continuous, $\sigma(\pi_E)$
determines $\sigma$.  Therefore the map
$\theta_{\pi_E}:\Aut_k(E)\ra\A(k)$ defined by
$\theta_{\pi_E}(\sigma)=\phi_{\sigma}$ is an
anti-isomorphism.  Let $\pi_E'$ be another uniformizer
for $E$; then there is a uniquely determined
$\lambda(t)\in\A(k)$ such that
$\pi_E'=\lambda|_{t=\pi_E}$.  For $\sigma\in\Aut_k(E)$
we get $\theta_{\pi_E'}(\sigma)
=\lambda\circ\theta_{\pi_E}(\sigma)\circ\lambda^{-1}$.

     Define the depth of $\phi(t)\in\A(k)$ to be
$D(\phi)=v_t(\phi(t)-t)-1$.  For $d\ge1$ set
\[\A(k)_d=\{\phi\in\A(k):D(\phi)\ge d\}.\]
Then $\A(k)_d$ is a normal subgroup of $\A(k)$; in
particular, we have $\N(k)=\A(k)_1$.  Let $\G$ be a finite
subgroup of $\A(k)$, and for $d\ge0$ set
\[\G_d=\{\phi\in\G:D(\phi)\ge d\}=\G\cap\A(k)_d.\] 
Then $\G_d$ is a normal subgroup of $\G$ which we call
the $d$th ramification subgroup of $\G$.  Say that
$b\ge0$ is a lower ramification break of $\G$ if
$\G_b\ne\G_{b+1}$.  If $b\ge1$ and $|\G_b:\G_{b+1}|=p^m$
we say that $b$ is a lower ramification break with
multiplicity $m$.  Suppose $|\G|=p^n$; then $\G$ has $n$
lower ramification breaks $b_1\le b_2\le\dots\le b_n$,
counted with multiplicities.  In some cases it is more
convenient to work with the upper ramification breaks of
$\G$ rather than the lower breaks.  In the case where
$|\G|=p^n$ these are defined by $u_1=b_1$ and
$u_{i+1}-u_i=p^{-i}(b_{i+1}-b_i)$ for $1\le i\le n-1$.
Note that while the lower ramification breaks of $\G$
are always integers, the upper breaks need not be.

     Let $F$ be a subfield of $E$ such that $E/F$ is a
finite totally ramified Galois extension with Galois
group $G=\Gal(E/F)$.  Set $\G=\theta_{\pi_E}(G)$, where
$\theta_{\pi_E}:\Aut_k(E)\ra\A(k)$ is the
anti-isomorphism determined by $\pi_E$.  Then the
ramification filtration on $\G$ induces a ramification
filtration $(G_b)_{b\ge0}$ on $G$.  Since
$\A(k)_d\trianglelefteq\A(k)$, the ramification
filtration on $G$ does not depend on the choice of
$\pi_E$.  We define lower and upper ramification breaks
for $G$ just as we did for $\G$; we also say that these
are the lower and upper ramification breaks of the
extension $E/F$.  Of course, these are the same as the
ramification breaks of $\G$.  These definitions agree
with the definitions of ramification jumps found, for
instance, in Chapter~IV of \cite{cl}.

     The following facts will be useful in
Sections~\ref{Qext} and \ref{Dext}:

\begin{prop} \label{breakfacts}
Let $E/F$ be a totally ramified Galois extension of
degree $p^n$ with lower ramification breaks
$b_1\le\dots\le b_n$ and upper ramification breaks
$u_1\le\dots\le u_n$.  Let $D/F$ be a subextension of
$E/F$.  Then
\begin{enumerate}[(a)]
\item The lower ramification breaks of $E/D$ are also
lower breaks of $E/F$.
\item If $D/F$ is Galois then the upper ramification
breaks of $D/F$ are also upper breaks of $E/F$.
\item Suppose $D$ is the fixed field of the ramification
subgroup $G_{b_i}$, where $i\ge2$ and $b_{i-1}<b_i$.
Then the lower breaks of $D/F$ are
$b_1\le\dots\le b_{i-1}$ and the upper breaks are
$u_1\le\dots\le u_{i-1}$.
\end{enumerate}
\end{prop}

\begin{proof}
Statement (a) follows directly from the definitions.
Statement (b) is a consequence of Proposition~14 of
\cite[IV]{cl}.  The first part of (c) follows from
Lemma~5 of \cite[IV]{cl}.  The second part then follows
from the definition of the upper breaks.
\end{proof}

\section{Field extensions and subgroups of the
Nottingham group}

In this section we formalize an idea that was used in
\cite{klopsch}, \cite{lubtor}, and \cite{BCT}.  Once again
$k\cong\F_q$ is a finite field of characteristic $p$.

     Let $E$ be a local field of characteristic $p$ with
residue field $k$.  Say $\sigma\in\Aut_k(E)$ is a wild
automorphism of $E$ if
$\sigma(\pi_E+\M_E^2)=\pi_E+\M_E^2$; this does not
depend on the choice of uniformizer $\pi_E$ for $E$.
The wild automorphisms of $E$ form a subgroup
$\Aut_k^1(E)$ of $\Aut_k(E)$ such that
$\theta_{\pi_E}(\Aut_k^1(E))=\N(k)$.  Suppose $E/F$ is a
finite totally ramified Galois extension of degree
$p^n$.  Then $\G=\theta_{\pi_E}(\Gal(E/F))$ is a
subgroup of the pro-$p$ group $\N(k)$, so
$\Gal(E/F)\le\Aut_k^1(E)$.

     Motivated in part by the notion of wild
automorphisms we make the following definition:
\begin{definition}
Let $k$ be a finite field of characteristic $p$.  We
define a category $\C_k$ whose objects are pairs
$(E/F,\pi_F+\M_F^2)$, where
\begin{enumerate}[(i)]
\item $F$ is a local field of characteristic $p$ with
residue field $k$,
\item $E/F$ is a finite totally ramified Galois
$p$-extension,
\item $\pi_F+\M_F^2$ is a generator for the
$\OO_F$-module $\M_F/\M_F^2$.
\end{enumerate}
A $\C_k$-morphism
\begin{equation} \label{morph}
\gamma:(E_1/F_1,\pi_{F_1}+\M_{F_1}^2)\lra
(E_2/F_2,\pi_{F_2}+\M_{F_2}^2)
\end{equation}
is a $k$-algebra isomorphism $\gamma:E_1\ra E_2$ such
that $\gamma(F_1)=F_2$ and
$\gamma(\pi_{F_1})+\M_{F_2}^2=\pi_{F_2}+\M_{F_2}^2$.
\end{definition}
\medskip

     It follows from the definition that every
$\C_k$-morphism is an isomorphism.  Let
$\gamma$ be the $\C_k$-morphism in
(\ref{morph}).  There is an isomorphism from
$\Gal(E_1/F_1)$ to $\Gal(E_2/F_2)$ which maps
$\sigma\in\Gal(E_1/F_1)$ to
$\gamma\circ\sigma\circ\gamma^{-1}$.  Furthermore, this
isomorphism respects the ramification filtrations of the
Galois groups.  Hence the ramification breaks of
$E_1/F_1$ and $E_2/F_2$ are the same.

     Let $E$ be a local field of characteristic $p$ with
residue field $k$ and let $\pi_E$ be a uniformizer for
$E$.  Let $\C^{E,\pi_E}$ denote the full subcategory of
$\C_k$ whose objects have the form
$(E/F,\pi_F+\M_F^2)$, where $E/F$ is a finite
totally ramified Galois $p$-extension
and $\pi_F=\norm_{E/F}(\pi_E)$.  A morphism
from $(E/F_1,\pi_{F_1}+\M_{F_1}^2)$ to
$(E/F_2,\pi_{F_2}+\M_{F_2}^2)$ is a wild
automorphism of $E$ which maps $F_1$ onto $F_2$.

     Let $K$ be a local field of characteristic $p$ with
residue field $k$ and let $\pi_K$ be a uniformizer for
$K$.  Let $\C_{K,\pi_K}$ denote the full subcategory of
$\C_k$ whose objects have the form
$(L/K,\pi_K+\M_K^2)$, where $L/K$ is a finite
totally ramified Galois $p$-extension.
A morphism from $(L_1/K,\pi_K+\M_K^2)$
to $(L_2/K,\pi_K+\M_K^2)$ is an isomorphism from
$L_1$ to $L_2$ which induces a wild automorphism of $K$.

\begin{prop} \label{equiv}
Let $K$ and $E$ be local fields of characteristic $p$
with residue field $k$.  Let $\pi_K$ be a uniformizer
for $K$ and let $\pi_E$ be a uniformizer for $E$.  Then
there are one-to-one correspondences between the
following:
\begin{enumerate}[(i)]
\item Isomorphism classes of objects in $\C_k$.
\item Isomorphism classes of objects in $\C_{K,\pi_K}$.
\item Isomorphism classes of objects in $\C^{E,\pi_E}$.
\item Conjugacy classes of finite subgroups of $\N(k)$.
\end{enumerate}
These correspondences map the isomorphism class of
$(L/K,\pi_K+\M_K^2)\in\C_{K,\pi_K}$ to the
conjugacy class represented by the subgroup
$\G=\theta_{\pi_L}(\Gal(L/K))$ of $\N(k)$, where $\pi_L$
is any uniformizer of $L$ such that
$\norm_{L/K}(\pi_L)+\M_K^2=\pi_K+\M_K^2$.
\end{prop}

\begin{proof}
It's clear that every object in $\C_k$ is isomorphic
to an object in $\C_{K,\pi_K}$, and to an object in
$\C^{E,\pi_E}$.  Hence the inclusions of $\C_{K,\pi_K}$
and $\C^{E,\pi_E}$ into $\C_k$ are equivalences of
categories.  This gives the one-to-one correspondences
between (i), (ii), and (iii).  Let $F_1,F_2$ be
subfields of $E$ such that $E/F_1$ and $E/F_2$ are
totally ramified Galois $p$-extensions.  Then
$(E/F_1,\pi_{F_1}+\M_{F_1}^2)$ is
$\C_k$-isomorphic to
$(E/F_2,\pi_{F_2}+\M_{F_2}^2)$ if and only if
there is a wild automorphism $\gamma$ of $E$ such that
$\gamma(F_1)=F_2$.  This is equivalent to
$\gamma\circ\Gal(E/F_1)\circ\gamma^{-1}=\Gal(E/F_2)$, so
the isomorphism classes in (iii) are in one-to-one
correspondence with conjugacy classes of finite
subgroups of $\Aut_k^1(E)$.  Since
$\theta_{\pi_E}:\Aut_k^1(E)\ra\N(k)$ is an
anti-isomorphism, this gives a one-to-one correspondence
between the isomorphism classes in (iii) and conjugacy
classes of finite subgroups of $\N(k)$.
\end{proof}

     Let $K$ be a local field of characteristic $p$ with
residue field $k$, let $\pi_K$ be a uniformizer for $K$,
and let $G$ be a finite $p$-group.  It follows from
Proposition~\ref{equiv} that conjugacy classes of
subgroups of $\N(k)$ which are isomorphic to $G$ are in
one-to-one correspondence with $\C_k$-isomorphism classes
of objects $(L/K,\pi_K+\M_K^2)$ such that $L/K$ is a
totally ramified $G$-extension.  The $G$-extensions
$L/K$ are relatively well-understood, since they can be
constructed iteratively using Artin-Schreier extensions.

\section{Elementary abelian extensions}
\label{elem}

Let $k=\F_4$ be the finite field with 4 elements and let
$s\in\F_4\smallsetminus\F_2$.  Then $s^2+s+1=0$ and
$\F_4=\{0,1,s,s^2\}$.  Let $K$ be a local field of
characteristic 2 with residue field $\F_4$.  In this
section we use Artin-Schreier theory to construct
totally ramified $(C_2\times C_2)$-extensions $M/K$ with
lower ramification breaks 1,1.  In Sections~\ref{Qext}
and \ref{Dext} we will use these extensions to construct
Galois extensions $L/K$ with Galois groups $Q_8$ and
$D_4$.

     In this section all fields are assumed to be
contained in a fixed algebraic closure $\Kb$ of $K$.
For each finite extension $L/K$ we let
$v_L:\Kb\ra\Q\cup\{\infty\}$ denote the valuation
normalized so that $v_L(L^{\times})=\Z$.  Since we are
working in characteristic 2, the Artin-Schreier operator
is given by $\wp(x)=x^p-x=x^2-x$.  

     The following elementary facts will be useful:

\begin{lemma} \label{ASval}
Let $L$ be a finite extension of $K$ and let
$\eta\in\Q$ with $\eta>0$.  Then for $a\in\Kb$ we have
$v_L(a)=-\eta$ if and only if $v_L(\wp(a))=-2\eta$.
\end{lemma}

\begin{lemma} \label{ASbreak}
Let $b$ be a positive integer such that $2\nmid b$ and
let $r\in K$ satisfy $v_K(r)=-b$.  Let $\alpha$ be a
root of $X^2-X-r$.  Then $K(\alpha)/K$ is a totally
ramified $C_2$-extension with upper and lower
ramification break $b$.  Furthermore, for every totally
ramified $C_2$-extension $L/K$ with ramification break
$b$ there is $r\in K$ with $v_K(r)=-b$ such that $L$ is
generated over $K$ by a root of $X^2-X-r$.
\end{lemma}

\begin{lemma} \label{2diff}
Let $a,b\in K$ and let $\alpha,\beta\in \Kb$ satisfy
$\alpha^2-\alpha=a$, $\beta^2-\beta=b$.  If
$K(\alpha)=K(\beta)$ then $a-b\in\wp(K)$.
\end{lemma}

\begin{proof}
Set $L=K(\alpha)=K(\beta)$.  Then $L/K$ is a Galois
extension of degree at most 2.  If $[L:K]=2$ and
$\sigma$ is the nonidentity element of $\Gal(L/K)$ then
\[\sigma(\alpha-\beta)=(\alpha+1)-(\beta+1)=\alpha-\beta.\]
Hence $a-b=\wp(\alpha-\beta)$ with $\alpha-\beta\in K$.
\end{proof}

     We now use Artin-Schreier theory to explicitly
construct totally ramified $(C_2\times C_2)$-extensions
of $K$ whose ramification breaks are as small as
possible.

\begin{prop} \label{Vclass}
Let $K$ be a local field of characteristic 2 with
residue field $\F_4$, and let $\pi_K$ be a uniformizer
for $K$.  There is a single $\C_{\F_4}$-isomorphism
class of objects
\[(M/K,\pi_K+\M_K^2)\in\C_{K,\pi_K}\]
such that $M/K$ is a totally ramified
$(C_2\times C_2)$-extension with lower (and upper)
ramification breaks 1,1.
\end{prop}

\begin{proof}
$(C_2\times C_2)$-extensions $M/K$ correspond under
Artin-Schreier theory to subgroups $H$ of $K/\wp(K)$
which are isomorphic to $C_2\times C_2$.  The extension
corresponding to $H$ is totally ramified if and only if
the image of $H$ in $K/(\OO_K+\wp(K))$ is also
isomorphic to $C_2\times C_2$.  Hence by
Lemma~\ref{ASbreak} totally ramified
$(C_2\times C_2)$-extensions $M/K$ with upper
ramification breaks 1,1 correspond to subgroups $H$ of
$\M_K^{-1}/(\wp(K)\cap\M_K^{-1})$ such that $H$ and the
image of $H$ in $\M_K^{-1}/\OO_K$ are both isomorphic to
$C_2\times C_2$.  Since
\[\wp(K)\cap\M_K^{-1}=\wp(\OO_K)=\wp(\F_4)+\wp(\M_K)
=\F_2+\M_K,\]
the quotient $\M_K^{-1}/(\wp(K)\cap\M_K^{-1})$ is an
elementary abelian 2-group of rank 3.  Since
$\M_K^{-1}/\OO_K$ is generated by
$\{s\pi_K^{-1}+\OO_K,s^2\pi_K^{-1}+\OO_K\}$,
each of the subgroups $H$ is generated by a uniquely
determined set of the form
\[S_{\epsilon_1,\epsilon_2}
=\{s\pi_K^{-1}+\epsilon_1+\wp(\OO_K),
s^2\pi_K^{-1}+\epsilon_2+\wp(\OO_K)\}\]
for some $\epsilon_1,\epsilon_2\in\{0,s\}$.  The
extension of $K$ which corresponds to this choice of $H$
is generated by the roots of
$X^2-X-(s\pi_K^{-1}+\epsilon_1)$ and
$X^2-X-(s^2\pi_K^{-1}+\epsilon_2)$.

     Let $H_0$ be the subgroup of $\M_K^{-1}/\wp(\OO_K)$
generated by $S_{0,0}$ and let $M_0/K$ be the
corresponding extension.  Then $M_0/K$ is a totally
ramified $(C_2\times C_2)$-extension with lower
ramification breaks 1,1.  Now let $M/K$ be another
totally ramified $(C_2\times C_2)$-extension with lower
ramification breaks 1,1, and let $H$ be the
corresponding subgroup of $\M_K^{-1}/\wp(\OO_K)$.  Then
$H$ is generated by $S_{\epsilon_1,\epsilon_2}$ for some
$\epsilon_1,\epsilon_2\in\{0,s\}$.  For $c\in\F_4$ let
$\gammab_c\in\Aut_{\F_4}^1(K)$ be defined by
$\gammab_c(\pi_K)=\pi_K+c\pi_K^2$.  Then
\begin{alignat*}{2}
\gammab_c(s\pi_K^{-1}+\epsilon_1)
&\equiv s\pi_K^{-1}+cs+\epsilon_1&&\pmod{\M_K} \\
\gammab_c(s^2\pi_K^{-1}+\epsilon_2)
&\equiv s^2\pi_K^{-1}+cs^2+\epsilon_2&&\pmod{\M_K}.
\end{alignat*}
There is $c_0\in\F_4$ such that
$\epsilon_1+c_0s,\epsilon_2+c_0s^2\in\F_2$.  Then
$\gammab_{c_0}(H)=H_0$, so $\gammab$ extends to an
isomorphism $\gamma:M\ra M_0$.  Therefore
$(M/K,\pi_K+\M_K^2)$ is $\C_{\F_4}$-isomorphic to
$(M_0/K,\pi_K+\M_K^2)$.
\end{proof}

\begin{remark}
The extension $M_0/K$ constructed in the proof of
Proposition~\ref{Vclass} depends on the choice of
uniformizer $\pi_K$ for $K$.
\end{remark}

\begin{lemma} \label{Y}
Let $M_0/K$ be the extension constructed in the proof of
Proposition~\ref{Vclass}.  Thus
$M_0=K(\alpha_1,\alpha_2)$ with
$\alpha_1^2-\alpha_1=s\pi_K^{-1}$ and
$\alpha_2^2-\alpha_2=s^2\pi_K^{-1}$.  Set
$Y=s\alpha_1+s^2\alpha_2$.  Then
\begin{enumerate}[(a)]
\item $\pi_K^{-1}=\wp(\wp(Y))=Y^4+Y$,
\item $\alpha_1=\wp(sY)=s^2Y^2+sY$,
\item $\alpha_2=\wp(s^2Y)=sY^2+s^2Y$.
\end{enumerate}
\end{lemma}

     It follows from Lemmas \ref{Y} and \ref{ASval} that
$v_{M_0}(Y)=-1$.  Therefore $\pi_{M_0}=Y^{-1}$ is a
uniformizer for $M_0$.  Define
$\sigmab_1,\sigmab_2\in\Gal(M_0/K)$ by
$\sigmab_1(\alpha_1)=\alpha_1+1$,
$\sigmab_1(\alpha_2)=\alpha_2$,
$\sigmab_2(\alpha_1)=\alpha_1$, and
$\sigmab_2(\alpha_2)=\alpha_2+1$.  Then
$\sigmab_1(Y)=Y+s$ and $\sigmab_2(Y)=Y+s^2$.

     In our applications, it will be enough to consider
extensions of $M_0$:

\begin{prop} \label{M0}
Let $K$ be a local field of characteristic 2 with
residue field $\F_4$, and let $\pi_K$ be a uniformizer
for $K$.  Let $G$ be a nonabelian group of order 8 and
let $L/K$ be a totally ramified $G$-extension with lower
ramification breaks $1,1,b$.  Then there is a
$G$-extension $L_0/K$ with $M_0\subset L_0$ such that
$(L/K,\pi_K+\M_K^2)$ is $\C_{\F_4}$-isomorphic to
$(L_0/K,\pi_K+\M_K^2)$.
\end{prop}

\begin{proof}
Let $M=L^{G_b}$ be the fixed field of $G_b$ acting on
$L$.  Using Proposition~\ref{breakfacts}(c) we see that
$M/K$ is a $(C_2\times C_2)$-extension with lower and
upper breaks 1,1.  It follows from
Proposition~\ref{Vclass} that there is an isomorphism
$\etab:M\ra M_0$ which induces a wild automorphism of
$K$.  There is a field $L_0$ containing $M_0$ such that
$\etab$ extends to an isomorphism $\eta:L\ra L_0$.  Then
$(L/K,\pi_K+\M_K^2)\cong(L_0/K,\pi_K+\M_K^2)$.
\end{proof}

     The following result will be used in
Section~\ref{Dext}:

\begin{prop} \label{Maut}
Let $M_0/K$ be the $(C_2\times C_2)$-extension
constructed in the proof of Proposition~\ref{Vclass} and
let $r\in\F_4$.  Then there is
$\rhob_r\in\Aut_{\F_4}^1(M_0)$ such that
$\rhob_r(\pi_K)=\pi_K+r\pi_K^3$,
$\rhob_r(Y)\equiv Y+r\pi_K\pmod{\M_{M_0}^8}$, and
$\rhob_r(\pi_{M_0})\equiv\pi_{M_0}+r\pi_{M_0}^6
\pmod{\M_{M_0}^9}$.
\end{prop}

\begin{proof}
Define $\rhoh_r\in\Aut_{\F_4}^1(K)$ by
$\rhoh_r(\pi_K)=\pi_K+r\pi_K^3$.  Then
\[\rhoh_r(\pi_K^{-1})\equiv\pi_K^{-1}+r\pi_K
\equiv\pi_K^{-1}+\wp(\wp(r\pi_K))\pmod{\M_K^2}.\]
Since $\M_K^2=\wp(\wp(\M_K^2))$, it follows from
Lemma~\ref{Y}(a) that $\rhoh_r$ extends to
$\rhob_r\in\Aut_{\F_4}^1(M_0)$ such that
\[\rhob_r(Y)\equiv Y+r\pi_K\pmod{\M_{M_0}^8}.\]
Since $\pi_{M_0}=Y^{-1}$ and
$\pi_K=1/(Y^4-Y)=\pi_{M_0}^4/(1-\pi_{M_0}^3)$ we get
\begin{alignat*}{2}
\rhob_r(\pi_{M_0})&\equiv\frac{\pi_{M_0}}{1+r\pi_{M_0}\pi_K}
&&\pmod{\M_{M_0}^{10}} \\
&\equiv\pi_{M_0}+r\pi_{M_0}^2\pi_K&&\pmod{\M_{M_0}^{10}} \\
&\equiv\pi_{M_0}+r\pi_{M_0}^6&&\pmod{\M_{M_0}^9}.
\end{alignat*}
\end{proof}

\section{Quaternion extensions}
\label{Qext}

Let $K$ be a local field of characteristic 2 with
residue field $\F_4$ and let $\pi_K$ be a uniformizer
for $K$.  In this section we classify isomorphism
classes of objects $(L/K,\pi_K+\M_K^2)\in\C_{K,\pi_K}$ such
that $L/K$ is a $Q_8$-extension with lower ramification
breaks 1,1,3.  As in Section~\ref{elem} we assume that
all extensions of $K$ are contained in an algebraic
closure $\Kb$ of $K$.

     Our first lemma shows that 1,1,3 are the smallest
possible lower breaks for a totally ramified
$Q_8$-extension.

\begin{lemma}
Let $L/K$ be a totally ramified Galois extension such
that $\Gal(L/K)\cong Q_8$.  Let $b_1\le b_2\le b_3$ be
the lower ramification breaks of $L/K$ and let
$u_1\le u_2\le u_3$ be the upper ramification breaks.
Then $b_1\ge1$, $b_2\ge1$, $b_3\ge3$, $u_1\ge1$,
$u_2\ge1$, and $u_3\ge\frac32$.
\end{lemma}

\begin{proof}
Since $L/K$ is a totally ramified Galois 2-extension
we must have $b_i\ge1$ and $u_i\ge1$ for all $i$.  The
bounds on $b_3$ and $u_3$ follow from Theorem~1.6 of
\cite{elder}.
\end{proof}

\begin{lemma} \label{perturb}
Let $L/K$ be a Galois $2$-extension and let $N$ be a
normal subgroup of $\Gal(L/K)$ of order $2$.  Let
$M=L^N$ be the fixed field of $N$, and let
$d\in M$ be such that $L$ is generated over $M$ by
a root of $X^2-X-d$.  Let $L'/K$ be a Galois extension
such that $M\subset L'$ and there is an isomorphism
$\theta:\Gal(L'/K)\ra\Gal(L/K)$ with
$\sigma|_{M}=\theta(\sigma)|_{M}$ for all
$\sigma\in\Gal(L'/K)$.  Then there is $\kappa\in K$ such
that $L'$ is generated over $M$ by a root of
$X^2-X-(d+\kappa)$.
\end{lemma}

\begin{proof}
This is proved as Lemma~1.8(b) in \cite{salt}, in the
more general setting of Galois extensions of rings of
characteristic $p>0$.
\end{proof}

     Let $M_0/K$ be the totally ramified
$(C_2\times C_2)$-extension constructed in
Proposition~\ref{Vclass}, and let
$Y=s\alpha_1+s^2\alpha_2$ be the generator for $M_0$
over $K$ defined in Lemma~\ref{Y}.  For $\kappa\in K$
let $L_{\kappa}$ be the extension of $M_0$ generated by
the roots of $X^2-X-Y^3-\kappa$.

\begin{prop} \label{Qram}
\begin{enumerate}[(a)]
\item Let $L/M_0$ be a finite extension.  Then $L/K$ is
a $Q_8$-extension if and only if $L=L_{\kappa}$ for some
$\kappa\in K$.
\item Let $\kappa\in\OO_K$.  Then $L_{\kappa}/K$ has
lower ramification breaks 1,1,3 and upper ramification
breaks $1,1,\frac32$.
\item Let $\kappa\in K$ be such that $r=-v_K(\kappa)$
satisfies $r\ge3$ and $2\nmid r$.  Then $L_{\kappa}/K$
has lower ramification breaks $1,1,1+4(r-1)$ and upper
ramification breaks $1,1,r$.
\end{enumerate}
\end{prop}

\begin{proof}
(a) Let $\kappa\in K$ and let $\alpha_3$ be a root of
$X^2-X-(Y^3+\kappa)$; then $L_{\kappa}=M_0(\alpha_3)$.
Since $\sigmab_1\in\Gal(M_0/K)$ satisfies
$\sigmab_1(Y)=Y+s$ we get
\begin{align*}
\wp(\alpha_3+s^2Y+s^2)
&=\wp(\alpha_3)+\wp(s^2Y)+\wp(s^2) \\
&=Y^3+\kappa+sY^2+s^2Y+1 \\
&=(Y+s)^3+\kappa \\
&=\sigmab_1(Y^3+\kappa).
\end{align*}
Therefore we may extend $\sigmab_1$ to
$\sigma_1\in\Gal(L_{\kappa}/K)$ by setting
\begin{equation} \label{sig1a3Q}
\sigma_1(\alpha_3)=\alpha_3+s^2Y+s^2.
\end{equation}
Similarly, we may extend $\sigmab_2$ to
$\sigma_2\in\Gal(L_{\kappa}/K)$ by setting
\begin{equation} \label{sig2a3Q}
\sigma_2(\alpha_3)=\alpha_3+sY+s.
\end{equation}
It follows that $L_{\kappa}/K$ is Galois, and that
\begin{align*}
\sigma_1(\sigma_2(\alpha_3))&=\sigma_1(\alpha_3+sY+s)
=\alpha_3+s^2Y+s^2+s(Y+s)+s=\alpha_3+Y+s \\
\sigma_2(\sigma_1(\alpha_3))&=\sigma_2(\alpha_3+s^2Y+s^2)
=\alpha_3+sY+s+s^2(Y+s^2)+s^2=\alpha_3+Y+s^2,
\end{align*}
Hence $\Gal(L_{\kappa}/K)$ is a nonabelian group of order 8.
Since
\begin{align*}
\sigma_1^2(\alpha_3)&=\sigma_1(\alpha_3+s^2Y+s^2)
=\alpha_3+s^2Y+s^2+s^2(Y+s)+s^2=\alpha_3+1 \\
\sigma_2^2(\alpha_3)&=\sigma_2(\alpha_3+sY+s)
=\alpha_3+sY+s+s(Y+s^2)+s=\alpha_3+1,
\end{align*}
$\sigma_1$ and $\sigma_2$ have order 4.  Hence
$\Gal(L_{\kappa}/K)\cong Q_8$.

     Now let $L/M_0$ be an extension such that
$L/K$ is a $Q_8$-extension.  Let $Z(Q_8)$ be the center
of $Q_8$.  The natural homomorphism
$\Aut(Q_8)\ra\Aut(Q_8/Z(Q_8))$ is onto, so there is an
isomorphism $\theta:\Gal(L/K)\ra\Gal(L_0/K)$ such that
$\sigma|_{M_0}=\theta(\sigma)|_{M_0}$ for all
$\sigma\in\Gal(L/K)$.  Hence by
Lemma~\ref{perturb} there exists $\kappa\in K$
such that $L=L_{\kappa}$.
\\[\smallskipamount]
(b) Since $v_{M_0}(Y^3+\kappa)=-3$, $L_{\kappa}/M_0$ is
a $C_2$-extension with lower ramification break 3.
Hence by Proposition~\ref{breakfacts}(a), 3 is a lower
break of $L_{\kappa}/K$.  Since 1,1 are upper breaks of
$M_0/K$, it follows from Proposition~\ref{breakfacts}(b)
that 1,1 are upper breaks of $L_{\kappa}/K$.  Hence 1,1
are lower breaks of $L_{\kappa}/K$.  Therefore
$L_{\kappa}/K$ has lower breaks 1,1,3 and upper breaks
$1,1,\frac32$.
\\[\smallskipamount]
(c) We have $\kappa\equiv c\pi_K^{-r}\pmod{\M_K^{-r+1}}$
for some $c\in\F_4^{\times}$.  Since $r\ge3$ and
$2\nmid r$ we get
\[\kappa\equiv c(Y^4-Y)^r\equiv cY^{4r}-cY^{4r-3}
\equiv-cY^{4r-3}+\wp(c^2Y^{2r})\pmod{\M_{M_0}^{-4r+4}}.\]
Let $\kappa'=\kappa-\wp(c^2Y^{2r})$.  Then $L_{\kappa}$
is generated over $M_0$ by a root of
$X^2-X-(Y^3+\kappa')$.  Since
$v_{M_0}(Y^3)=-3>-4r+3=v_{M_0}(\kappa')$,
$L_{\kappa}/M_0$ is a $C_2$-extension with lower break
$-v_{M_0}(Y^3+\kappa')=1+4(r-1)$.  It now follows by the
argument used in the proof of (b) that $L_{\kappa}/K$
has lower breaks $1,1,1+4(r-1)$ and upper breaks
$1,1,r$.
\end{proof}

\begin{theorem} \label{embedQ}
\begin{enumerate}[(a)]
\item There are precisely two extensions $L/M_0$ such
that $L/K$ is a totally ramified $Q_8$-extension with
lower ramification breaks $1,1,3$ and upper ramification
breaks $1,1,\frac32$.  These are the fields $L_0,L_s$
generated over $M_0$ by the roots of
$X^2-X-(Y^3+\delta)$, with $\delta\in\{0,s\}$.
\item Let $L/K$ be a totally ramified $Q_8$-extension
with lower ramification breaks $1,1,3$.  Then
$(L/K,\pi_K+\M_K^2)$ is $\C_{\F_4}$-isomorphic either to
$(L_0/K,\pi_K+\M_K^2)$, or to $(L_s/K,\pi_K+\M_K^2)$.
\item The objects $(L_0/K,\pi_K+\M_K^2)$,
$(L_s/K,\pi_K+\M_K^2)$ are not $\C_{\F_4}$-isomorphic.
\end{enumerate}
\end{theorem}

\begin{proof}
(a) Let $\delta\in\{0,s\}$.  Then by
Proposition~\ref{Qram} $L_{\delta}/K$ is a
$Q_8$-extension with lower ramification breaks $1,1,3$
and upper ramification breaks $1,1,\frac32$.  Now let
$L/M_0$ be an extension such that $L/K$ is a
$Q_8$-extensions with lower ramification breaks $1,1,3$.
It follows from Proposition~\ref{Qram}(a) that
$L=L_{\kappa}$ for some $\kappa\in K$.  Assume without
loss of generality that $\kappa$ is chosen so that
\begin{equation} \label{max}
v_K(\kappa)
=\max\{v_K(\kappa'):\kappa'\in(\kappa+\wp(M_0))\cap K\}.
\end{equation}
Suppose $v_K(\kappa)<0$.  If $2\mid v_K(\kappa)$ then
there is $\lambda\in K$ such that
$v_K(\lambda)=\frac12v_K(\kappa)$ and
$v_K(\kappa+\wp(\lambda))>v_K(\kappa)$.  This
contradicts (\ref{max}), so we have $2\nmid v_K(\kappa)$.
If $v_K(\kappa)=-r$ with $r\ge3$ and $2\nmid r$ then it
follows from Proposition~\ref{Qram}(c) that $r$ is an
upper ramification break of $L/K$, contrary to
assumption.  If $v_K(\kappa)=-1$ then $\kappa\equiv
c_1s\pi_K^{-1}+c_2s^2\pi_K^{-1}\pmod{\OO_K}$ for some
$c_1,c_2\in\F_2$.  Then
$\kappa+\wp(c_1\alpha_1+c_2\alpha_2)\in\OO_K$, again
contradicting (\ref{max}).  Hence we have
$\kappa\in\OO_K$.  Since $0,s$ are coset representatives
for $\OO_K/\wp(\OO_K)$ we can take $\kappa\in\{0,s\}$.
\\[\medskipamount]
(b) By Proposition~\ref{M0} $(L/K,\pi_K+\M_K^2)$ is
$\C_{\F_4}$-isomorphic to $(L'/K,\pi_K+\M_K^2)$ for some
totally ramified $Q_8$-extension $L'/K$ such that
$M_0\subset L'$.  Since $L/K$ has lower ramification
breaks 1,1,3, so does $L'/K$.  Hence by (a), $L'$ must
be one of $L_0,L_s$.
\\[\medskipamount]
(c) Suppose
\[\gamma:(L_0/K,\pi_K+\M_K^2)\lra(L_s/K,\pi_K+\M_K^2)\]
is a $\C_{\F_4}$-isomorphism.  Then
$\gamma:L_0\ra L_s$ is a field isomorphism which
restricts to a wild automorphism of $K$.
Furthermore, since $M_0/K$
is the unique $(C_2\times C_2)$-subextension of both
$L_0/K$ and $L_s/K$ we have $\gamma(M_0)=M_0$.  Write
$\gamma(\pi_K)=\pi_K+a_2\pi_K^2+a_3\pi_K^3+\cdots$ with
$a_i\in\F_4$.  Then $M_0=\gamma(M_0)=K(Y')$, with
\[(Y')^4-Y'=\frac{1}{\gamma(\pi_K)}
=\frac{1}{\pi_K+a_2\pi_K^2+a_3\pi_K^3+\cdots}
=\pi_K^{-1}(1+a_2\pi_K+(a_2^2+a_3)\pi_K^2+\cdots).\]
Hence $M_0$ contains the root $Y'-Y$ of
\[X^4-X=\gamma(\pi_K)^{-1}-\pi_K^{-1}
=a_2+(a_2^2+a_3)\pi_K+\cdots.\]
Suppose $a_2\not=0$.  Since $\wp(\wp(\M_K))=\M_K$ this
implies that $M_0/K$ has a nontrivial unramified
subextension, a contradiction.  Suppose $a_2=0$.  Then
we can assume that $Y'=Y+\epsilon$ for some
$\epsilon\in\M_K$.  Therefore
$(Y')^3\equiv Y^3\pmod{\M_{M_0}}$.  Since $L_0$ is
generated over $M_0$ by a root of $X^2-X-Y^3$, and $L_s$
is generated over $M_0$ by a root of $X^2-X-(Y')^3$,
this implies $L_0=L_s$.  Hence $L_0$ contains a root of
$X^2-X-s$, another contradiction.
\end{proof}

     By applying Proposition~\ref{equiv} we get the
following:

\begin{cor} \label{Qgroups}
There are two conjugacy classes of subgroups of
$\N(\F_4)$ which are isomorphic to $Q_8$ and have lower
ramification breaks 1,1,3.
\end{cor}

\section{Quaternion subgroups of $\N(\F_4)$}
\label{Qsub}

In this section we use the $Q_8$-extensions constructed
in Section~\ref{Qext} to give explicit descriptions of
subgroups $\QQ_8^0$, $\QQ_8^s$ of $\N(\F_4)$ in terms of
finite automata.  These subgroups are representatives
of the two conjugacy classes of minimally ramified
$Q_8$-subgroups of $\N(\F_4)$ given by
Corollary~\ref{Qgroups}.

     Let $L_0/K$ be the $Q_8$-extension constructed in
Theorem~\ref{embedQ}(a).  Then $L_0=M_0(\alpha_3)$,
where $\alpha_3$ is a root of $X^2-X-Y^3$.  Since
$v_{M_0}(Y^3)=-3$, it follows from Lemma~\ref{ASval}
that $v_{L_0}(\alpha_3)=-3$.  Set $t=Y\alpha_3^{-1}$.
Then $v_{L_0}(t)=-2-(-3)=1$, so $t$ is a uniformizer for
$L_0$.  We have
$\norm_{L_0/M_0}(t)=Y^2(Y^3)^{-1}=Y^{-1}$, and hence
$\norm_{L_0/K}(t)=\norm_{M_0/K}(Y^{-1})=\pi_K$.
It follows that the group
$\QQ_8^0=\{\sigma(t):\sigma\in\Gal(L_0/K)\}$ is a
representative of the conjugacy class of subgroups of
$\N(\F_4)$ which corresponds under
Proposition~\ref{equiv} to the object
$(L_0/K,\pi_K+\M_K^2)\in\C_{K,\pi_K}$.

     It follows from the definitions that $Y$,
$\alpha_3$, and $t$ are algebraic over the function
field $\F_4(\pi_K)$.  By applying $\sigma_1,\sigma_2$ to
the equation $\alpha_3t=Y$, and using (\ref{sig1a3Q}),
(\ref{sig2a3Q}), we get
\begin{align*}
(\alpha_3+s^2Y+s^2)\sigma_1(t)&=Y+s \\
(\alpha_3+sY+s)\sigma_2(t)&=Y+s^2.
\end{align*}
Therefore $\sigma_1(t),\sigma_2(t)$ are also algebraic
over $\F_4(\pi_K)$.  It follows that for $i=1,2$ there
is a polynomial relation between $t$ and $\sigma_i(t)$
defined over $\F_4$.  To determine these relations we
view $Y,\alpha_3,t$ as variables and let $X$ represent
$\sigma_i(t)$.  Consider the following ideals in the
polynomial ring $\F_4[Y,\alpha_3,t,X]$, which encode
relations between elements of $L_0$:
\begin{align*}
J_1&=(\alpha_3^2-\alpha_3-Y^3,t\alpha_3-Y,
(\alpha_3+s^2Y+s^2)X-Y-s) \\
J_2&=(\alpha_3^2-\alpha_3-Y^3,t\alpha_3-Y,
(\alpha_3+sY+s)X-Y-s^2).
\end{align*}
Using Magma \cite{magma}, we compute a Gr\"obner basis
for $J_i$ with a term order which eliminates $Y$ and
$\alpha_3$.  The Gr\"obner basis includes a polynomial
$g_{\sigma_i}(t,X)$ which does not depend on $Y$ or
$\alpha_3$ and has $X=\sigma_i(t)$ as a root.  The
polynomial $g_{\sigma_i}(t,X)$ factors nontrivially, but
it has a unique irreducible factor $f_{\sigma_i}(t,X)$
which admits a root of the form $X=t+a_2t^2+\cdots$ in
$\F_4[[t]]$.  

     We find that $X=\sigma_i(t)$ is a root of
$f_{\sigma_i}(t,X)$, where
\begin{align*}
f_{\sigma_1}(t,X)&=(t^2+1)X^2+X+st^2+t \\
f_{\sigma_2}(t,X)&=(t^2+1)X^2+X+s^2t^2+t.
\end{align*}
Since $f_{\sigma_i}(0,0)=0$ and
$\frac{\partial f_{\sigma_i}}{\partial X}(0,0)\not=0$,
the polynomial $f_{\sigma_i}(t,X)$ is ``non-singular''
in the sense of Section 3.2.3 of \cite{BCT}.  This
allows us to use the algorithm given there to compute an
automaton which determines the coefficients of
$\sigma_i(t)$.  This algorithm, which we have
implemented in Magma, is based on Furstenberg's theorem
\cite{Furst}.  Since $f_{\sigma_2}(t,X)$ comes from
$f_{\sigma_1}(t,X)$ by applying the 2-Frobenius to
$\F_4$, the automata for $\sigma_1(t)$ and $\sigma_2(t)$
have the same digraph and edge labels, and the state
labels for $\sigma_2(t)$ are the
$\Gal(\F_4/\F_2)$-conjugates of the state labels for
$\sigma_1(t)$:
\[\begin{array}{|c||c|c|c|c||c|c|}
\hline
\text{State}&0&1&2&3&\sigma_1\text{ label}
&\sigma_2\text{ label} \\
\hline
1&2&3&4&5&0&0 \\
\hline
2&2&6&7&4&0&0 \\
\hline
3&3&5&5&5&1&1 \\
\hline
4&6&8&7&4&s^2&s \\
\hline
5&5&5&5&5&0&0 \\
\hline
6&6&6&7&4&s^2&s \\
\hline
7&8&8&7&4&s&s^2 \\
\hline
8&8&6&7&4&s&s^2 \\
\hline
\end{array}\]
A diagram of the automaton for $\sigma_1(t)$ is given in
Figure~\ref{aut1}.  Note that directed edges leading
from a state back to itself are not included in our
diagrams.

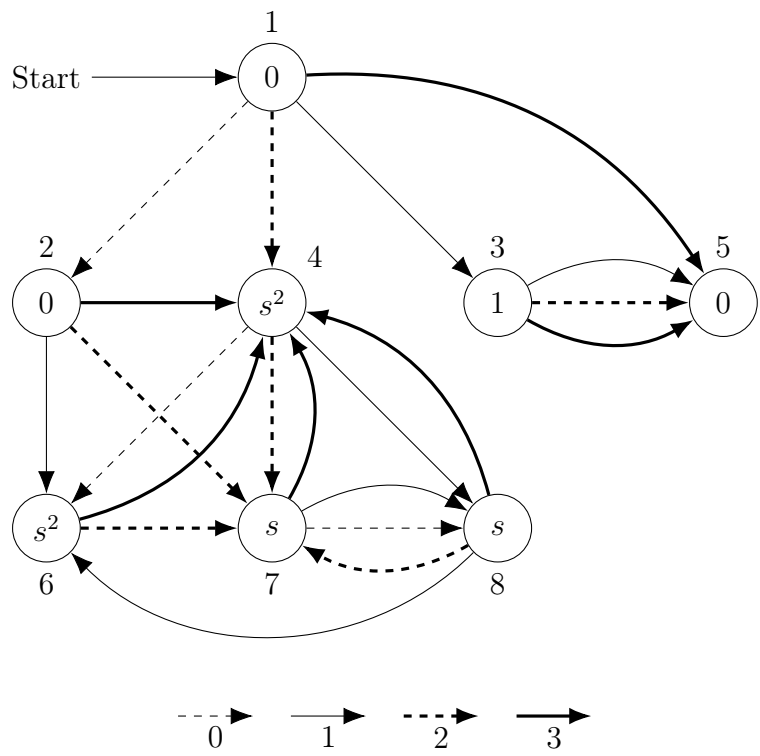
\begin{figure}
\caption{Automaton for $\sigma_1$}
\label{aut1}
\begin{center}
\begin{tikzpicture}[node distance = 3cm, auto]
\node[label={1},draw,circle,minimum width=9mm] at (0,0) (1) {$0$};
\node[left of=1] (0) {Start};
\node[label={above right:4},draw,circle,minimum width=9mm,below of=1] (4) {$s^2$};
\node[label={2},draw,circle,minimum width=9mm,left of=4] (2) {$0$};
\node[label={3},draw,circle,minimum width=9mm,right of=4] (3) {$1$};
\node[label={5},draw,circle,minimum width=9mm,right of=3] (5) {$0$};
\node[label={below:6},draw,circle,minimum width=9mm,below of=2] (6) {$s^2$};
\node[label={below:7},draw,circle,minimum width=9mm,right of=6] (7) {$s$};
\node[label={below:8},draw,circle,minimum width=9mm,right of=7] (8) {$s$};
\draw[-{Latex[length=3mm]}] (0) to (1);
\draw[-{Latex[length=3mm]},dashed,very thick] (1) to (4);
\draw[-{Latex[length=3mm]},dashed,thin] (1) to (2);
\draw[-{Latex[length=3mm]},thin] (1) to (3);
\draw[-{Latex[length=3mm]},bend left,very thick] (1) to (5);
\draw[-{Latex[length=3mm]},thin] (4) to (8);
\draw[-{Latex[length=3mm]}] (2) to (6);
\draw[-{Latex[length=3mm]},dashed,very thick] (2) to (7);
\draw[-{Latex[length=3mm]},bend left=45,thin] (8) to (6);
\draw[-{Latex[length=3mm]},bend left,thin] (3) to (5);
\draw[-{Latex[length=3mm]},bend right,very thick] (3) to (5);
\draw[-{Latex[length=3mm]},dashed,very thick] (3) to (5);
\draw[-{Latex[length=3mm]},very thick] (2) to (4);
\draw[-{Latex[length=3mm]},bend right,very thick] (6) to (4);
\draw[-{Latex[length=3mm]},bend right,very thick] (7) to (4);
\draw[-{Latex[length=3mm]},bend right,very thick] (8) to (4);
\draw[-{Latex[length=3mm]},dashed,very thick] (4) to (7);
\draw[-{Latex[length=3mm]},dashed,very thick,bend left] (8) to (7);
\draw[-{Latex[length=3mm]},dashed,very thick] (6) to (7);
\draw[-{Latex[length=3mm]},dashed,thin] (4) to (6);
\draw[-{Latex[length=3mm]},dashed,thin] (7) to (8);
\draw[-{Latex[length=3mm]},bend left,thin] (7) to (8);
\draw[-{Latex[length=3mm]},dashed,thin] (-1.25,-8.5) to
node[swap]{0} (-.25,-8.5);
\draw[-{Latex[length=3mm]},thin] (.25,-8.5) to
node[swap]{1} (1.25,-8.5);
\draw[-{Latex[length=3mm]},dashed,very thick] (1.75,-8.5) to
node[swap]{2} (2.75,-8.5);
\draw[-{Latex[length=3mm]},very thick] (3.25,-8.5) to
node[swap]{3} (4.25,-8.5);
\end{tikzpicture}
\end{center}
\end{figure}

     Let $\sigma_0=\sigma_1\sigma_2$.  Automata for
$\sigma_1^3(t)$, $\sigma_2^3(t)$, $\sigma_0(t)$,
$\sigma_0^3(t)$, and $\sigma_0^2(t)$ can also be
computed.  As above we use Gr\"obner bases to get
polynomials
\begin{align*}
f_{\sigma_1^3}(t,X)&=(t^2+s)X^2+X+t^2+t \\
f_{\sigma_2^3}(t,X)&=(t^2+s^2)X^2+X+t^2+t \\
f_{\sigma_0}(t,X)&=(t^2+s^2)X^2+X+st^2+t \\
f_{\sigma_0^3}(t,X)&=(t^2+s)X^2+X+s^2t^2+t \\
f_{\sigma_0^2}(t,X)&=t^2X^2+X+t
\end{align*}
which have the specified series as roots.  All these
polynomials are non-singular, so we can apply the
algorithm in Section~3.2.3 of \cite{BCT} in each case.
The automata for $\sigma_1^3(t)$, $\sigma_2^3(t)$,
$\sigma_0(t)$, and $\sigma_0^3(t)$ all have the same
graph and the same edge labels, but different state
labels:
\[\begin{array}{|c||c|c|c|c||c|c|c|c|c|c|}
\hline
\text{State}&0&1&2&3&\sigma_1^3\text{ label}
&\sigma_2^3\text{ label}&\sigma_0\text{ label}
&\sigma_0^3\text{ label} \\
\hline
1&2&3&4&5&0&0&0&0 \\
\hline
2&2&6&7&8&0&0&0&0 \\
\hline
3&3&5&5&5&1&1&1&1 \\
\hline
4&9&10&11&4&s^2&s&1&1 \\
\hline
5&5&5&5&5&0&0&0&0 \\
\hline
6&6&9&12&13&s&s^2&s&s^2 \\
\hline
7&14&10&11&4&1&1&s&s^2 \\
\hline
8&15&16&7&8&1&1&s^2&s \\
\hline
9&9&15&11&4&s^2&s&1&1 \\
\hline
10&10&6&7&8&s&s^2&1&1 \\
\hline
11&16&14&12&13&s^2&s&s^2&s \\
\hline
12&10&16&7&8&s&s^2&1&1 \\
\hline
13&6&14&12&13&s&s^2&s&s^2 \\
\hline
14&14&15&11&4&1&1&s&s^2 \\
\hline
15&15&6&7&8&1&1&s^2&s \\
\hline
16&16&9&12&13&s^2&s&s^2&s \\
\hline
\end{array}\]
A diagram of the automaton for $\sigma_1^3(t)$ is
given in Figure~\ref{aut13}.
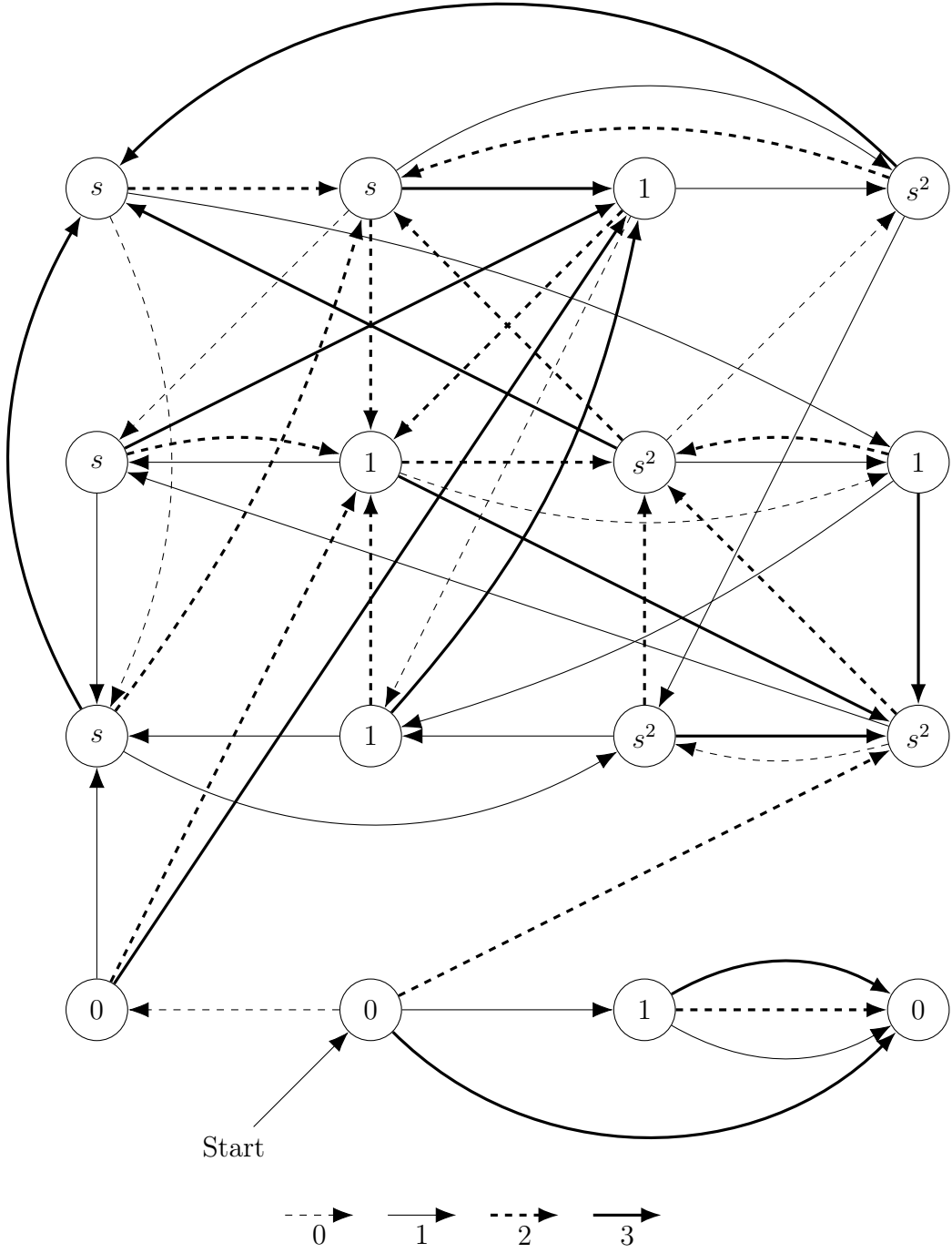
\begin{figure}
\caption{Automaton for $\sigma_1^3$}
\label{aut13}
\begin{center}
\begin{tikzpicture}[auto]
\node (0) at (-2,-10) {Start};
\node[draw,circle,minimum width=9mm] at (0,-8) (1) {$0$};
\node[draw,circle,minimum width=9mm] at (-4,-8) (2) {$0$};
\node[draw,circle,minimum width=9mm] at (4,-8) (3) {$1$};
\node[draw,circle,minimum width=9mm] at (8,-4) (4) {$s^2$};
\node[draw,circle,minimum width=9mm] at (8,-8) (5) {$0$};
\node[draw,circle,minimum width=9mm] at (-4,-4) (6) {$s$};
\node[draw,circle,minimum width=9mm] at (0,0) (7) {$1$};
\node[draw,circle,minimum width=9mm] at (4,4) (8) {$1$};
\node[draw,circle,minimum width=9mm] at (4,-4) (9) {$s^2$};
\node[draw,circle,minimum width=9mm] at (-4,0) (10) {$s$};
\node[draw,circle,minimum width=9mm] at (4,0) (11) {$s^2$};
\node[draw,circle,minimum width=9mm] at (0,4) (12) {$s$};
\node[draw,circle,minimum width=9mm] at (-4,4) (13) {$s$};
\node[draw,circle,minimum width=9mm] at (8,0) (14) {$1$};
\node[draw,circle,minimum width=9mm] at (0,-4) (15) {$1$};
\node[draw,circle,minimum width=9mm] at (8,4) (16) {$s^2$};
\draw[-{Latex[length=3mm]}] (0) to (1);
\draw[-{Latex[length=3mm]},dashed,thin] (1) to (2);
\draw[-{Latex[length=3mm]},dashed,thin,bend left=15] (4) to (9);
\draw[-{Latex[length=3mm]},dashed,thin,bend right=20] (7) to (14);
\draw[-{Latex[length=3mm]},dashed,thin] (8) to (15);
\draw[-{Latex[length=3mm]},dashed,thin] (11) to (16);
\draw[-{Latex[length=3mm]},dashed,thin] (12) to (10);
\draw[-{Latex[length=3mm]},dashed,thin,bend left=25] (13) to (6);
\draw[-{Latex[length=3mm]},thin] (1) to (3);
\draw[-{Latex[length=3mm]},thin] (2) to (6);
\draw[-{Latex[length=3mm]},thin,bend right] (3) to (5);
\draw[-{Latex[length=3mm]},thin] (4) to (10);
\draw[-{Latex[length=3mm]},thin,bend right] (6) to (9);
\draw[-{Latex[length=3mm]},thin] (7) to (10);
\draw[-{Latex[length=3mm]},thin] (8) to (16);
\draw[-{Latex[length=3mm]},thin] (9) to (15);
\draw[-{Latex[length=3mm]},thin] (10) to (6);
\draw[-{Latex[length=3mm]},thin] (11) to (14);
\draw[-{Latex[length=3mm]},thin,bend left=35] (12) to (16);
\draw[-{Latex[length=3mm]},thin,bend left=10] (13) to (14);
\draw[-{Latex[length=3mm]},thin,bend left=10] (14) to (15);
\draw[-{Latex[length=3mm]},thin] (15) to (6);
\draw[-{Latex[length=3mm]},thin] (16) to (9);
\draw[-{Latex[length=3mm]},dashed,very thick] (1) to (4);
\draw[-{Latex[length=3mm]},dashed,very thick] (2) to (7);
\draw[-{Latex[length=3mm]},dashed,very thick] (3) to (5);
\draw[-{Latex[length=3mm]},dashed,very thick] (4) to (11);
\draw[-{Latex[length=3mm]},dashed,very thick,bend right=10] (6) to (12);
\draw[-{Latex[length=3mm]},dashed,very thick] (7) to (11);
\draw[-{Latex[length=3mm]},dashed,very thick] (8) to (7);
\draw[-{Latex[length=3mm]},dashed,very thick] (9) to (11);
\draw[-{Latex[length=3mm]},dashed,very thick,bend left=15] (10) to (7);
\draw[-{Latex[length=3mm]},dashed,very thick] (11) to (12);
\draw[-{Latex[length=3mm]},dashed,very thick] (12) to (7);
\draw[-{Latex[length=3mm]},dashed,very thick] (13) to (12);
\draw[-{Latex[length=3mm]},dashed,very thick,bend right=15] (14) to (11);
\draw[-{Latex[length=3mm]},dashed,very thick] (15) to (7);
\draw[-{Latex[length=3mm]},dashed,very thick,bend right=20] (16) to (12);
\draw[-{Latex[length=3mm]},very thick,bend right=45] (1) to (5);
\draw[-{Latex[length=3mm]},very thick] (2) to (8);
\draw[-{Latex[length=3mm]},very thick,bend left] (3) to (5);
\draw[-{Latex[length=3mm]},very thick,bend left] (6) to (13);
\draw[-{Latex[length=3mm]},very thick] (7) to (4);
\draw[-{Latex[length=3mm]},very thick] (9) to (4);
\draw[-{Latex[length=3mm]},very thick] (10) to (8);
\draw[-{Latex[length=3mm]},very thick] (11) to (13);
\draw[-{Latex[length=3mm]},very thick] (12) to (8);
\draw[-{Latex[length=3mm]},very thick] (14) to (4);
\draw[-{Latex[length=3mm]},very thick,bend right=15] (15) to (8);
\draw[-{Latex[length=3mm]},very thick,bend right=45] (16) to (13);
\draw[-{Latex[length=3mm]},dashed,thin] (-1.25,-11) to
node[swap]{0} (-.25,-11);
\draw[-{Latex[length=3mm]},thin] (.25,-11) to
node[swap]{1} (1.25,-11);
\draw[-{Latex[length=3mm]},dashed,very thick] (1.75,-11) to
node[swap]{2} (2.75,-11);
\draw[-{Latex[length=3mm]},very thick] (3.25,-11) to
node[swap]{3} (4.25,-11);
\end{tikzpicture}
\end{center}
\end{figure}
The automaton for $\sigma_1^2(t)$ is given by the
following table, and is diagrammed in
Figure~\ref{aut12}.
\[\begin{array}{|c||c|c|c|c||c|}
\hline
\text{State}&0&1&2&3&\sigma_1^2\text{ label} \\
\hline
1&2&3&4&5&0 \\
\hline
2&5&3&5&5&0 \\
\hline
3&3&5&5&5&1 \\
\hline
4&5&2&3&4&0 \\
\hline
5&5&5&5&5&0 \\
\hline
\end{array}\]

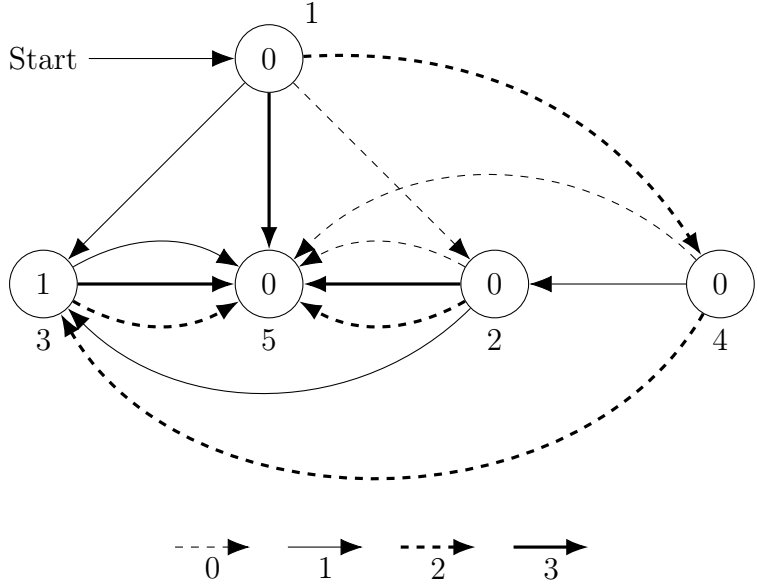
\begin{figure}
\caption{Automaton for $\sigma_1^2$}
\label{aut12}
\begin{center}
\begin{tikzpicture}[node distance = 3cm, auto]
\node[label={above right:1},draw,circle,minimum width=9mm] at (0,0) (1) {$0$};
\node[left of=1] (0) {Start};
\node[label={below:5},draw,circle,minimum width=9mm,below of=1] (5) {$0$};
\node[label={below:3},draw,circle,minimum width=9mm,left of=5] (3) {$1$};
\node[label={below:2},draw,circle,minimum width=9mm,right of=5] (2) {$0$};
\node[label={below:4},draw,circle,minimum width=9mm,right of=2] (4) {$0$};
\draw[-{Latex[length=3mm]}] (0) to (1);
\draw[-{Latex[length=3mm]},dashed,thin] (1) to (2);
\draw[-{Latex[length=3mm]},thin] (1) to (3);
\draw[-{Latex[length=3mm]},bend left,dashed,very thick] (1) to (4);
\draw[-{Latex[length=3mm]},very thick] (1) to (5);
\draw[-{Latex[length=3mm]},bend right,dashed,thin] (2) to (5);
\draw[-{Latex[length=3mm]},very thick] (2) to (5);
\draw[-{Latex[length=3mm]},bend left,dashed,very thick] (2) to (5);
\draw[-{Latex[length=3mm]},bend left=45,thin] (2) to (3);
\draw[-{Latex[length=3mm]},bend left,thin] (3) to (5);
\draw[-{Latex[length=3mm]},bend right,dashed,very thick] (3) to (5);
\draw[-{Latex[length=3mm]},very thick] (3) to (5);
\draw[-{Latex[length=3mm]},bend right=45,dashed,thin] (4) to (5);
\draw[-{Latex[length=3mm]},thin] (4) to (2);
\draw[-{Latex[length=3mm]},bend left=60,dashed,very thick] (4) to (3);
\draw[-{Latex[length=3mm]},dashed,thin] (-1.25,-6.5) to
node[swap]{0} (-.25,-6.5);
\draw[-{Latex[length=3mm]},thin] (.25,-6.5) to
node[swap]{1} (1.25,-6.5);
\draw[-{Latex[length=3mm]},dashed,very thick] (1.75,-6.5) to
node[swap]{2} (2.75,-6.5);
\draw[-{Latex[length=3mm]},very thick] (3.25,-6.5) to
node[swap]{3} (4.25,-6.5);
\end{tikzpicture}
\end{center}
\end{figure}

     We now describe a subgroup $\QQ_8^s$ of $\N(\F_4)$
which corresponds to the $Q_8$-extension $L_s/K$ from
Proposition~\ref{embedQ}.  Rather than generating $L_s$
over $M_0$ with a root $\alpha_3$ of $X^2-X-Y^3-s$, we
let $\alphah_3$ be a root of $X^2-X-Y^3-sY^2-s^2Y-s$.
Since $sY^2+s^2Y=\wp(s^2Y)$ we can take
$\alphah_3=\alpha_3+s^2Y$; thus $L_s=M_0(\alphah_3)$.
This choice of generator for $L_s$ leads to a subgroup
of $\N(\F_4)$ which is generated by elements
represented by automata with relatively small numbers of
states.  Using (\ref{sig1a3Q}) and (\ref{sig2a3Q}) we
extend $\sigmab_1,\sigmab_2$ to elements of
$\Gal(L_s/K)$ such that
$\sigma_1(\alphah_3)=\alphah_3+s^2Y+s$ and
$\sigma_2(\alphah_3)=\alphah_3+sY$.  Also set
$\sigma_0=\sigma_1\sigma_2$; then
$\sigma_0(\alphah_3)=\alphah_3+Y+1$.  Let
$t=Y\alphah_3^{-1}$; then $t$ is a uniformizer for
$L_s$.  We have $\norm_{L_s/M_0}(t)=Y^2(Y^3+s)^{-1}
\equiv Y^{-1}\pmod{\M_{M_0}^2}$, and hence
$\norm_{L_s/K}(t)\equiv\norm_{M_0/K}(Y^{-1})\equiv\pi_K
\pmod{\M_K^2}$.  Therefore the group
$\QQ_8^s=\{\sigma(t):\sigma\in\Gal(L_0/K)\}$ is a
representative of the conjugacy class of subgroups of
$\N(\F_4)$ which corresponds to 
$(L_s/K,\pi_K+\M_K^2)\in\C_{K,\pi_K}$.

     As above, for each of the series $\sigma(t)$
corresponding to the nonidentity elements of $Q_8$ we
use Gr\"{o}bner bases to compute irreducible polynomials
$f_{\sigma}(t,X)\in\F_4[t,X]$ such that $X=\sigma(t)$ is
a root of $f_{\sigma}(t,X)$.  For instance,
$f_{\sigma_1},f_{\sigma_2}$ are elements of Gr\"{o}bner
bases for the ideals
\begin{align*}
J_1&=(\alphah_3^2-\alphah_3-Y^3-sY^2-s^2Y-s,
t\alpha_3-Y,(\alpha_3+s^2Y+s)X-Y-s) \\
J_2&=(\alphah_3^2-\alphah_3-Y^3-sY^2-s^2Y-s,
t\alpha_3-Y,(\alpha_3+sY)X-Y-s^2)
\end{align*}
in the polynomial ring $\F_4[Y,\alphah_3,t,X]$.  We find
that
\begin{align*}
f_{\sigma_1}&=(t^3 + s^2t^2 + t + s)X^3 + (s^2t^3 + st^2
+ 1)X^2 + (t^3 + s^2t^2 + t + s)X + st^3 + t^2 + st \\
f_{\sigma_1^3}&=(t^3 + s^2t^2 + t + s)X^3 + (s^2t^3 +
st^2 + s^2t + 1)X^2 + (t^3 + t + s)X + st^3 + t^2 + st \\
f_{\sigma_2}&=(t^3 + s^2t^2 + st + 1)X^3 + (t^2 + s)X^2
+ (t^3 + s^2t^2 + st + 1)X + t^3 + st^2 + t \\
f_{\sigma_2^3}&=(t^3 + t + 1)X^3 + (s^2t^3 + t^2 + s^2t
+ s)X^2 + (st^3 + st + 1)X + t^3 + st^2 + t \\
f_{\sigma_0}&=(t^3 + s^2t^2 + 1)X^3 + (t^3 + s^2t^2 +
s^2)X^2 + (t^3 + s^2t^2 + s^2t + s^2)X + t^3 + s^2t^2 + s^2t \\
f_{\sigma_0^3}&=(t^3 + t^2 + t + 1)X^3 + (s^2t^3 + s^2t^2
+ s^2t + s^2)X^2 + (s^2t + s^2)X + t^3 + s^2t^2 + s^2t \\
f_{\sigma_0^2}&=(t^3+st^2+s^2t+s)X^3 + (st^3+s^2t)X^2 +
(s^2t^3+s^2t^2)X + st^3.
\end{align*}

     The first six of these polynomials are nonsingular,
so the corresponding automata can be computed using the
algorithm in Section~3.2.3 of \cite{BCT}.  We get 1773
states for $\sigma_1(t)$, 36 states for $\sigma_1^3(t)$,
263 states for $\sigma_2(t)$, 203 states for
$\sigma_2^3(t)$, 595 states for $\sigma_0(t)$, and 95
states for $\sigma_0^3(t)$.  Since $f_{\sigma_0^2}$ is
singular we use the algorithm in Section~3.2.2 of
\cite{BCT} to compute an automaton for $\sigma_0^2(t)$.
This algorithm, which is based on Ore polynomials, gives
an automaton with 1768 states.  The
automata for $\sigma_1^3(t)$ and $\sigma_0^3(t)$ are
given in Appendices~\ref{sig2p} and \ref{sig3p}.
Together, these two elements generate the subgroup
$\QQ_8^s$ of $\N(\F_4)$.

\begin{remark}
Since the composita $L_0\F_{16}$, $L_s\F_{16}$ are
equal, our $Q_8$-subgroups $\QQ_8^0,\QQ_8^s$ of
$\N(\F_4)$ are conjugate in $\N(\F_{16})$.  Therefore it
may seem surprising that the automata associated to
$L_s/K$ are so much more complicated than those
associated to $L_0/K$.
\end{remark}

\section{Dihedral extensions} \label{Dext}

Let $K$ be a local field of characteristic 2 with
residue field $\F_4$ and let $\pi_K$ be a uniformizer
for $K$.  In this section we classify isomorphism
classes of objects $(L/K,\pi_K+\M_K^2)\in\C_{K,\pi_K}$ such
that $L/K$ is a $D_4$-extension with lower ramification
breaks 1,1,5.  Once again we assume that all extensions
of $K$ are contained in the algebraic closure $\Kb$ of
$K$.

     We begin by observing that 1,1,5 are the smallest
possible lower breaks for a totally ramified
$D_4$-extension.

\begin{lemma}
Let $L/K$ be a totally ramified Galois extension such
that $\Gal(L/K)\cong D_4$.  Let $b_1\le b_2\le b_3$ be
the lower ramification breaks of $L/K$ and let
$u_1\le u_2\le u_3$ be the upper ramification breaks.
Then $b_1\ge1$, $b_2\ge1$, $b_3\ge5$, $u_1\ge1$,
$u_2\ge1$, and $u_3\ge2$.
\end{lemma}

\begin{proof}
Since $L/K$ is a totally ramified Galois $p$-extension
we must have $b_i\ge1$ and $u_i\ge1$ for $i=1,2$.  The
bounds on $b_3$ and $u_3$ follow from Theorem~1.6 of
\cite{elder}.
\end{proof}

     Let $M_0/K$ be the totally ramified
$(C_2\times C_2)$-extension constructed in
Proposition~\ref{Vclass}, and let
$Y=s\alpha_1+s^2\alpha_2$ be the generator for $M_0$
over $K$ defined in Lemma~\ref{Y}.  Let
\[g(X)=X^5+X^4+X^3\in\F_4[X].\]
For $\zeta\in\F_4^{\times}$ and $\kappa\in K$ let
$L_{\zeta,\kappa}$ be the extension of $M_0$ generated
by the roots of $X^2-X-g(\zeta Y)-\kappa$.

\begin{prop} \label{Dram}
\begin{enumerate}[(a)]
\item Let $L/M_0$ be a finite extension.  Then $L/K$ is
a $D_4$-extension if and only if $L=L_{\zeta,\kappa}$
for some $\zeta\in\F_4^{\times}$ and $\kappa\in K$.
\item Let $\zeta\in\F_4^{\times}$ and $\kappa\in\OO_K$.
Then $L_{\zeta,\kappa}/K$ has lower ramification breaks
1,1,5 and upper ramification breaks 1,1,2.
\item Let $\zeta\in\F_4^{\times}$ and $\kappa\in K$.
Assume that $r:=-v_K(\kappa)$ satisfies $r\ge3$ and
$2\nmid r$.  Then $L_{\zeta,\kappa}/K$ has lower
ramification breaks $1,1,1+4(r-1)$ and upper
ramification breaks $1,1,r$.
\end{enumerate}
\end{prop}

\begin{proof}
(a) Let $\zeta\in\F_4^{\times}$, let $\kappa\in K$, and
let $\alpha_3^{\zeta,\kappa}$ be a root of
$X^2-X-(g(\zeta Y)+\kappa)$; then
$L_{\zeta,\kappa}=M_0(\alpha_3^{\zeta,\kappa})$.  Set
$\sigmab_0=\sigmab_1\sigmab_2$, where
$\sigmab_1,\sigmab_2$ are the generators of
$\Gal(M_0/K)$ defined in Section~\ref{elem}.  By
applying the methods used in the proof of
Proposition~\ref{Qram}(a) we can extend
$\sigmab_1,\sigmab_2,\sigmab_0\in\Gal(M_0/K)$ to
$\tau_1^{\zeta,\kappa},\tau_2^{\zeta,\kappa},
\tau_0^{\zeta,\kappa}\in\Gal(L_{\zeta,\kappa}/K)$ by
setting
\begin{alignat}{3}
\tau_1^{1,\kappa}(\alpha_3^{1,\kappa})
&=\alpha_3^{1,\kappa}+s^2Y^2+Y \nonumber \\
\tau_2^{1,\kappa}(\alpha_3^{1,\kappa})
&=\alpha_3^{1,\kappa}+sY^2+Y \label{tau1} \\
\tau_0^{1,\kappa}(\alpha_3^{1,\kappa})
&=\alpha_3^{1,\kappa}+Y^2+s^2 \nonumber \\[3mm]
\tau_1^{s,\kappa}(\alpha_3^{s,\kappa})&=\alpha_3^{s,\kappa}+Y^2+sY
\nonumber \\
\tau_2^{s,\kappa}(\alpha_3^{s,\kappa})&=\alpha_3^{s,\kappa}+s^2Y^2+s^2
\\
\tau_0^{s,\kappa}(\alpha_3^{s,\kappa})&=\alpha_3^{s,\kappa}+sY^2+sY
\nonumber \\[3mm]
\tau_1^{s^2,\kappa}(\alpha_3^{s^2,\kappa})&=\alpha_3^{s^2,\kappa}+sY^2+s
\nonumber \\
\tau_2^{s^2,\kappa}(\alpha_3^{s^2,\kappa})&=\alpha_3^{s^2,\kappa}+Y^2+s^2Y
\label{taus2} \\
\tau_0^{s^2,\kappa}(\alpha_3^{s^2,\kappa})&=\alpha_3^{s^2,\kappa}+s^2Y^2+s^2Y
\nonumber
\end{alignat}
Hence $L_{\zeta,\kappa}/K$ is a Galois extension.  Let
$0\le i\le2$ and $\zeta\in\F_4^{\times}$.  If
$(i,\zeta)\in\{(0,1),(1,s^2),(2,s)\}$ then
$(\tau_i^{\zeta,\kappa})^2(\alpha_3^{1,\kappa})
=\alpha_3^{\zeta,\kappa}+1$, and hence
$\tau_i^{\zeta,\kappa}$ has order 4.  Thus
$L_{\zeta,\kappa}\not=M_0$, so $L_{\zeta,\kappa}/K$ is
an extension of degree 8.  If
$(i,\zeta)\not\in\{(0,1),(1,s^2),(2,s)\}$ then
$(\tau_i^{\zeta,\kappa})^2(\alpha_3^{1,\kappa})
=\alpha_3^{\zeta,\kappa}$, so $\tau_i^{\zeta,\kappa}$ has
order 2.  It follows that for each
$\zeta\in\F_4^{\times}$ there is exactly one
$0\le i\le2$ such that $\sigmab_i$ lifts to an element
of $\Gal(L_{\zeta,\kappa}/K)$ with order 4.  Therefore
$\Gal(L_{\zeta,\kappa}/K)\cong D_4$.

     Let $L/M_0$ be an extension such that $L/K$ is a
$D_4$-extension.  Let $\tau\in\Gal(L/K)$ have order 4.
Then there exists $0\le i\le2$ such that $\tau$
restricts to $\sigmab_i\in\Gal(M_0/K)$, and
$\zeta\in\F_4^{\times}$ such that
$\tau_i^{\zeta,0}\in\Gal(L_{\zeta,0}/K)$
has order 4.  There is an automorphism of $D_4$ which
fixes every rotation and switches the two cosets in
$D_4/Z(D_4)$ whose elements are reflections.  Hence
there is an isomorphism
$\theta:\Gal(L_{\zeta,0}/K)\ra\Gal(L/K)$ such that
$\theta(\tau_i^{\zeta,0})=\tau$ and
$\sigma|_{M_0}=\theta(\sigma)|_{M_0}$ for all
$\sigma\in\Gal(L/K)$.  It now follows from
Lemma~\ref{perturb} that there exists $\kappa\in K$ such
that $L=L_{\zeta,\kappa}$.
\\[\smallskipamount]
(b)-(c) The arguments used in the proofs of parts (b)
and (c) of Proposition~\ref{Qram} are valid here as
well.
\end{proof}

\begin{theorem} \label{embedD}
\begin{enumerate}[(a)]
\item There are precisely six extensions $L/M_0$ such
that $L/K$ is a totally ramified $D_4$-extension with
lower ramification breaks $1,1,5$ and upper ramification
breaks $1,1,2$, namely $L_{\zeta,\delta}$ with
$\zeta\in\F_4^{\times}$ and $\delta\in\{0,s\}$.
\item For every totally ramified $D_4$-extension $L/K$
with lower ramification breaks $1,1,5$, the object
$(L/K,\pi_K+\M_K^2)$ is $\C_{\F_4}$-isomorphic to
$(L_{\zeta,0}/K,\pi_K+\M_K^2)$ for some
$\zeta\in\F_4^{\times}$.  In particular,
$(L_{\zeta,s}/K,\pi_K+\M_K^2)$ is $\C_{\F_4}$-isomorphic
to $(L_{\zeta,0}/K,\pi_K+\M_K^2)$.
\item The objects $(L_{\zeta,0}/K,\pi_K+\M_K^2)$ with
$\zeta\in\F_4^{\times}$ are pairwise nonisomorphic.
\end{enumerate}
\end{theorem}

\begin{proof}
(a) Let $\zeta\in\F_4^{\times}$ and $\delta\in\{0,s\}$.
Then by Proposition~\ref{Dram} $L_{\zeta,\delta}/K$ is a
$D_4$-extension with lower ramification breaks $1,1,5$
and upper ramification breaks 1,1,2.  Now let
$L/M_0$ be an extension such that $L/K$ is a
totally ramified $D_4$-extension with lower ramification
breaks $1,1,5$.  It follows from
Proposition~\ref{Dram}(a) that $L=L_{\zeta,\kappa}$ for
some $\zeta\in\F_4^{\times}$ and $\kappa\in K$.  Assume
without loss of generality that $\kappa$ is chosen to
satisfy (\ref{max}).  As in the proof of
Theorem~\ref{embedQ}(a) we get $\kappa\in\OO_K$.
Since $\{0,s\}$ are coset representatives for
$\OO_K/\wp(\OO_K)$ we can take $\kappa\in\{0,s\}$.
Suppose $\zeta,\zeta'\in\F_4^{\times}$ and
$\delta,\delta'\in\{0,s\}$ satisfy
$L_{\zeta,\delta}=L_{\zeta',\delta'}$.  Then by
Lemma~\ref{2diff} we have $g(\zeta'Y)+\delta'
=g(\zeta Y)+\delta+\wp(\mu)$ for some $\mu\in M_0$.  It
follows that
$(\zeta'-\zeta)Y^5\equiv\wp(\mu)\pmod{\M_{M_0}^{-4}}$,
so $\zeta'=\zeta$.  Hence $\delta'-\delta=\wp(\mu)$, so
$\delta'=\delta$.  We conclude that the six
$D_4$-extensions $L_{\zeta,\delta}/K$ are distinct.
\\[\medskipamount]
(b) By Proposition~\ref{M0} $(L/K,\pi_K+\M_K^2)$ is
$\C_{\F_4}$-isomorphic to $(L'/K,\pi_K+\M_K^2)$ for some
totally ramified $D_4$-extension $L'/K$ such that
$M_0\subset L'$.  Since $L/K$ has lower ramification
breaks 1,1,5, so does $L'/K$.  Hence by (a), $L'$ must
be equal to $L_{\zeta,\delta}$ for some
$\zeta\in\F_4^{\times}$, $\delta\in\{0,s\}$.  If
$\delta=0$ then we are done, so we focus on the case
$\delta=s$.

     Let $r\in\F_4$.  Then by Proposition~\ref{Maut}
there is $\rhob_r\in\Aut_{\F_4}(M_0)$ such that
$\rhob_r(\pi_K)=\pi_K+r\pi_K^3$ and
$\rhob_r(Y)\equiv Y+r\pi_K \pmod{\M_{M_0}^8}$.  Hence
\begin{alignat*}{2}
\rhob_r(Y^5)&\equiv(Y+r\pi_K)^5&&\pmod{\M_{M_0}} \\
&\equiv Y^5+r\pi_KY^4&&\pmod{\M_{M_0}} \\
&\equiv Y^5+r&&\pmod{\M_{M_0}} \\[2mm]
\rhob_r(Y^4)&\equiv(Y+r\pi_K)^4&&\pmod{\M_{M_0}} \\
&\equiv Y^4&&\pmod{\M_{M_0}} \\[2mm]
\rhob_r(Y^3)&\equiv(Y+r\pi_K)^3&&\pmod{\M_{M_0}} \\
&\equiv Y^3&&\pmod{\M_{M_0}}.
\end{alignat*}
Let $\alpha_3^{\zeta,s}$ be a root of
$X^2-X-g(\zeta Y)-s$; then
$L_{\zeta,s}=M_0(\alpha_3^{\zeta,s})$.  It
follows from the congruences above that by setting
$r=\zeta s$ we get $\rhob_r(g(\zeta Y)+s)
=g(\zeta Y)+\epsilon$ for some
$\epsilon\in\M_{M_0}=\wp(\M_{M_0})$.  Hence $\rhob_r$
extends to an isomorphism
$\rho_r:L_{\zeta,s}\ra L_{\zeta,0}$ such that
$\rho_r(\pi_K)=\pi_K+r\pi_K^3$.  Thus
$\rho_r$ induces a wild automorphism of $K$, so
\[\rho_r:(L_{\zeta,s}/K,\pi_K+\M_K^2)\lra
(L_{\zeta,0}/K,\pi_K+\M_K^2)\]
is a $\C_{\F_4}$-isomorphism.  It follows that
$(L/K,\pi_K+\M_K^2)$ is $\C_{\F_4}$-isomorphic to
$(L_{\zeta,0}/K,\pi_K+\M_K^2)$.
\\[\medskipamount]
(c) We will show that $(L_{1,0}/K,\pi_K+\M_K^2)$ and
$(L_{s,0}/K,\pi_K+\M_K^2)$ are not $\C_{\F_4}$-isomorphic;
the other cases are similar.  Suppose
\[\gamma:(L_{1,0}/K,\pi_K+\M_K^2)\lra
(L_{s,0}/K,\pi_K+\M_K^2)\]
is a $\C_{\F_4}$-isomorphism.  Then $\gamma(M_0)=M_0$.
Let $\alpha_3^{1,0}\in L_{1,0}$ be a root of
$X^2-X-g(Y)$, let $\alpha_3^{s,0}\in L_{s,0}$ be a root
of $X^2-X-g(sY)$, and let $\tau_3^{s,0}$ be the
non-identity element of $\Gal(L_{s,0}/M_0)$.  Then
$L_{s,0}=M_0(\gamma(\alpha_3^{1,0}))=M_0(\alpha_3^{s,0})$
and
\[\tau_3^{s,0}(\gamma(\alpha_3^{1,0}))-\gamma(\alpha_3^{1,0})
=\tau_3^{s,0}(\alpha_3^{s,0})-\alpha_3^{s,0}=1.\]
Hence $\gamma(\alpha_3^{1,0})-\alpha_3^{s,0}\in M_0$.
Since $\gamma$ induces a wild automorphism of $M_0$ we
have $\gamma(Y)\equiv Y\pmod{\OO_{M_0}}$, and hence
$\gamma(g(Y))\equiv Y^5\pmod{\M_{M_0}^{-4}}$.  We also
have $g(sY)\equiv s^2Y^5\pmod{\M_{M_0}^{-4}}$, which
gives
\[\gamma(g(Y))-g(sY)\equiv sY^5\pmod{\M_{M_0}^{-4}}.\]  
It follows that
\[\wp(\gamma(\alpha_3^{1,0})-\alpha_3^{s,0})
=\gamma(g(Y))-g(sY)\not\in\wp(M_0),\]
a contradiction.  Therefore our two objects are not
$\C_{\F_4}$-isomorphic.
\end{proof}

     Once again, Proposition~\ref{equiv} allows us to
translate our classification of certain isomorphism
classes of objects in $\C_{\F_4}$ into a statement about
subgroups of the Nottingham group:

\begin{cor} \label{Dgroups}
There are three conjugacy classes of subgroups of
$\N(\F_4)$ which are isomorphic to $D_4$ and have lower
ramification breaks 1,1,5.
\end{cor}

\section{Dihedral subgroups of $\N(\F_4)$}

In this section we use the $D_4$-extensions constructed
in Section~\ref{Dext} to give explicit descriptions of
certain subgroups of $\N(\F_4)$ in terms of finite
automata.  These subgroups are representatives for the
three conjugacy classes of minimally ramified
$D_4$-subgroups of $\N(\F_4)$ given by
Corollary~\ref{Dgroups}.

     Let $K$ be a local field of characteristic 2 with
residue field $\F_4$ and let $\pi_K$ be a uniformizer
for $K$.  Let $M_0=K(Y)$ be the
$(C_2\times C_2)$-extension constructed in
Proposition~\ref{Vclass}.  It follows from
Theorem~\ref{embedD} that there are three
$\C_{\F_4}$-isomorphism classes of objects
$(L/K,\pi_K+\M_K^2)$ such that $L/K$ is a $D_4$-extension
with lower ramification breaks 1,1,5.  We can assume
that $M_0\subset L$, and that our three extensions are
generated over $M_0$ by roots
$\alpha_3^{1,0},\alpha_3^{s,0},\alpha_3^{s^2,0}$ of the
polynomials
\begin{align*}
X^2-X-g(Y)&=X^2-X-(Y^5+Y^4+Y^3) \\
X^2-X-g(sY)&=X^2-X-(s^2Y^5+sY^4+Y^3) \\
X^2-X-g(s^2Y)&=X^2-X-(sY^5+s^2Y^4+Y^3).
\end{align*}
We denote these fields by $L_{1,0}=K(\alpha_3^{1,0})$,
$L_{s,0}=K(\alpha_3^{s,0})$, and
$L_{s^2,0}=K(\alpha_3^{s^2,0})$.  In the proof of
Proposition~\ref{Dram}(a) we extended the elements
$\sigmab_1,\sigmab_2,\sigmab_0$ of $\Gal(M_0/K)$ to
elements $\tau_1^{\zeta,0},\tau_2^{\zeta,0},
\tau_3^{\zeta,0}$ of $\Gal(L_{\zeta,0}/K)$ for each
$\zeta\in\F_4^{\times}$.  Viewed as elements of $D_4$,
$\tau_0^{1,0},\tau_2^{s,0},\tau_1^{s^2,0}$ are rotations
of order 4, and $\tau_1^{1,0},\tau_2^{1,0},\tau_1^{s,0},
\tau_0^{s,0},\tau_2^{s^2,0},\tau_0^{s^2,0}$ are
reflections.  Using Lemma~\ref{ASval} we get
$v_{L_{\zeta,0}}(\alpha_3^{\zeta,0})=-5$.  Hence
$t_{\zeta}:=\zeta Y^2(\alpha_3^{\zeta,0})^{-1}$
satisfies
\[v_{L_{\zeta,0}}(t_{\zeta})
=2v_{L_{\zeta,0}}(Y)-v_{L_{\zeta,0}}(\alpha_3^{\zeta,0})
=2(-2)-(-5)=1.\]
Therefore $t_{\zeta}$ is a uniformizer for
$L_{\zeta,0}$.  We have
\[\norm_{L_{\zeta,0}/M_0}(t_{\zeta})
=\zeta^2Y^4g(\zeta Y)^{-1}\equiv Y^{-1}
\pmod{\M_{M_0}^2},\]
and hence
\[\norm_{L_{\zeta,0}/K}(t)\equiv\norm_{M_0/K}(Y^{-1})
\equiv\pi_K\pmod{\M_K^2}.\]
It follows that the group
\[\D_4^{\zeta}
=\{\sigma(t):\sigma\in\Gal(L_{\zeta,0}/K)\}\]
is a representative of the conjugacy class of subgroups
of $\N(\F_4)$ which corresponds under
Proposition~\ref{equiv} to the object
$(L_{\zeta,0}/K,\pi_K+\M_K^2)\in\C_{K,\pi_K}$.

     By applying $\tau_1^{\zeta,0},\tau_2^{\zeta,0}$ to
the equation $\alpha_3^{\zeta,0}t_{\zeta}=\zeta Y^2$,
and using (\ref{tau1})--(\ref{taus2}), we get
\begin{align*}
(\alpha_3^{1,0}+s^2Y^2+Y)\tau_1^{1,0}(t_1)&=Y^2+s^2 \\
(\alpha_3^{1,0}+sY^2+Y)\tau_2^{1,0}(t_1)&=Y^2+s \\
(\alpha_3^{s,0}+Y^2+sY)\tau_1^{s,0}(t_s)&=sY^2+1 \\
(\alpha_3^{s,0}+s^2Y^2+s^2)\tau_2^{s,0}(t_s)&=sY^2+s^2 \\
(\alpha_3^{s^2,0}+sY^2+s)\tau_1^{s^2,0}(t_{s^2})&=s^2Y^2+s \\
(\alpha_3^{s^2,0}+Y^2+s^2Y)\tau_2^{s^2,0}(t_{s^2})&=s^2Y^2+1.
\end{align*}
As in Section~\ref{Qsub} we use Gr\"obner bases to
compute irreducible polynomials
$f_{\tau_i^{\zeta,0}}(t,X)$ which have
$X=\tau_i^{\zeta,0}(t)$ as a root:
\begin{align*}
f_{\tau_1^{1,0}}(t,X)&=(t+s^2)X^3 + (t^2+s^2t)X^2 +
(t^3+s^2t^2+s^2t+1)X+(s^2t^3+t) \\
f_{\tau_2^{1,0}}(t,X)&=(t+s)X^3 + (t^2+st)X^2 +
(t^3+st^2+st+1)X+(st^3+t) \\
f_{\tau_1^{s,0}}(t,X)
&=(t+1)X^3+(t^2+t)X^2+(t^3+t^2+s^2t+1)x+(t^3+t) \\
f_{\tau_2^{s,0}}(t,X)
&=(st^2+t+s^2)X^3+(t^3+s^2t)X^2
+(s^2t^3+st^2+s^2t+s)X+(st^3+st) \\
f_{\tau_1^{s^2,0}}(t,X)
&=(s^2t^2+t+s)X^3+(t^3+st)X^2
+(st^3+s^2t^2+st+s^2)X+(s^2t^3+s^2t) \\
f_{\tau_2^{s^2,0}}(t,X)
&=(t+1)X^3+(t^2+t)X^2+(t^3+t^2+st+1)x+(t^3+t).
\end{align*}

     Each of these polynomials is nonsingular, so we can
use the algorithm from Section~3.2.3 of \cite{BCT} to
compute the corresponding automata.  Since
$f_{\tau_1^{1,0}}(t,X)$ and $f_{\tau_2^{1,0}}(t,X)$ are
$\Gal(\F_4/\F_2)$ conjugates of each other, the automata
for $\tau_1^{1,0}(t)$ and $\tau_2^{1,0}(t)$ have the
same digraph, with the same edge labels and conjugate
vertex labels.  These automata are given in
Appendix~\ref{3table}.  Since $f_{\tau_1^{s,0}}(t,X),
f_{\tau_2^{s^2,0}}(t,X)$ are $\Gal(\F_4/\F_2)$
conjugates, the automata for
$\tau_1^{s,0}(t),\tau_2^{s^2,0}(t)$ have the same
digraph, the same edge labels, and conjugate vertex
labels.  These are given in Appendix~\ref{reflecttable}.
Similarly, $f_{\tau_2^{s,0}}(t,X),f_{\tau_1^{s^2,0}}(t,X)$
are $\Gal(\F_4/\F_2)$ conjugates, so the automata for
$\tau_2^{s,0}(t),\tau_1^{s^2,0}(t)$ have the same
digraph, the same edge labels, and conjugate vertex
labels, which we give in Appendix~\ref{rottable}.
All six of these automata have 104 states.
Note that for each $\zeta\in\F_4^{\times}$,
$\{\tau_1^{\zeta,0}(t),\tau_2^{\zeta,0}(t)\}$ is a
generating set for the group $\D_4^{\zeta}$ associated
to the extension $L_{\zeta,0}/K$.

\begin{remark}
There are $\F_4$-isomorphisms between $L_{1,0}$,
$L_{s,0}$, and $L_{s^2,0}$ which map $K$ onto $K$.
Therefore the three $D_4$-subgroups of $\N(\F_4)$
constructed here are all conjugate in $\A(\F_4)$.
\end{remark}

\appendix

\section{Automata for generators of $\QQ_8^s$}

\subsection{Automaton for $\sigma_1^3$}
\label{sig2p}
\vspace{1mm}
\[\begin{array}{|c||c|c|c|c||c|||c||c|c|c|c||c|}
\hline
\text{State}&0&1&2&3&\sigma_1^3\text{ label}
&\text{State}&0&1&2&3&\sigma_1^3\text{ label} \\
\hline
 1& 2& 3& 4& 5& 0&19& 7&25&29& 5&s^2 \\ \hline
 2& 6& 7& 8& 9& 0&20&25&14&13&15& s \\ \hline
 3&10&11&12&11& 1&21&14&12&18&12& 1 \\ \hline
 4&11& 3&10&13&s^2&22& 7&10&29&10&s^2 \\ \hline
 5& 8&14& 9&15& 0&23&23&17& 5&28& 0 \\ \hline
 6&16&17&18&18& 0&24&12& 8&11&30& s \\ \hline
 7&19&20&21&22&s^2&25&20&31&22&32& s \\ \hline
 8& 8&23& 9&24& 0&26&11&12&10&12&s^2 \\ \hline
 9&17&25&15& 5& 0&27&12&10&11&10& s \\ \hline
10& 3&26&24&18& 1&28&10&23&12&24& 1 \\ \hline
11&26&27&28&29&s^2&29&31&31&32&32& 1 \\ \hline
12&27& 3&30&13& s&30&11&17&10&28&s^2 \\ \hline
13&19&19&21&21&s^2&31&14& 7&18& 9& 1 \\ \hline
14&31&19&32&21& 1&32&25&11&13&11& s \\ \hline
15&23& 7& 5& 9& 0&33&34&11&35&11& 0 \\ \hline
16&33&14& 7&22& 0&34&36&23&12&21& 0 \\ \hline
17&17& 8&15&30& 0&35&17&26&15&18& 0 \\ \hline
18&20&20&22&22& s&36& 2& 3& 4& 5& 0 \\ \hline
\end{array}\]

\subsection{Automaton for $\sigma_0^3$}
\label{sig3p}
\footnotesize
\vspace{4mm}
\[\begin{array}{|c||c|c|c|c||c|||c||c|c|c|c||c|}
\hline
\text{State}&0&1&2&3&\sigma_0\text{ label}
&\text{State}&0&1&2&3&\sigma_0^3\text{ label} \\
\hline
1& 2& 3& 4& 5& 0&49&72&81&74&82& 1 \\ \hline
2& 6& 7& 8& 9& 0&50&17&83&16&84& 1 \\ \hline
3&10&11&12&13& 1&51& 7&10&17&12& 0 \\ \hline
4&10&14&12&15& 1&52&11& 7&84&16& s \\ \hline
5& 7&16&17& 7& 0&53&85&85&77&74&s^2 \\ \hline
6&18&19&20&21& 0&54&85& 9&77&17&s^2 \\ \hline
7&22&23&24&25& 0&55&81&70&82&77& s \\ \hline
8&26&27&28&29& 1&56&75&73&80&71& s \\ \hline
9&23&24&25&30& s&57&75&10&80&12& s \\ \hline
 10&26&21&28&31& 1&58&14&83&76&84&s^2 \\ \hline
 11&29&27&32&29& s&59&79&78&15&86& 0 \\ \hline
 12&33&33&34&28&s^2&60&83&73&86&71&s^2 \\ \hline
 13&35&21&36&31&s^2&61& 9&85& 9&74& s \\ \hline
 14&27&37&29&34&s^2&62& 9&79& 9&80& s \\ \hline
 15&19&29&21&38& 1&63&14&72&76& 8&s^2 \\ \hline
 16&24&22&38&21&s^2&64&10&17&12& 9& 1 \\ \hline
 17&39&39&40&36& 1&65&83&16&86& 7&s^2 \\ \hline
 18&41&42&43&44& 0&66&81&11&82&13& s \\ \hline
 19&45&46&47&48& 1&67&73&17&71& 9& 0 \\ \hline
 20&49&45&44&50& 1&68&78& 9&13&17& 1 \\ \hline
 21& 3&51& 5&52& 1&69&87&26&88&34& 0 \\ \hline
 22&51&53&52&54& 0&70&37&29&89&38& 0 \\ \hline
 23&55&56&49&57& s&71&21&37&31&34& 1 \\ \hline
 24&58&45&42&50&s^2&72&21&19&31&90& 1 \\ \hline
 25&51& 3&52& 5& 0&73&91&92&93&40& 0 \\ \hline
 26&49&59&44&60& 1&74&27&92&29&40&s^2 \\ \hline
 27&42&49&60&61&s^2&75&92&35&94&94& s \\ \hline
 28&45&56&47&57& 1&76&91&19&93&90& 0 \\ \hline
 29&56&62&50&63& s&77&90&90&30&32& 0 \\ \hline
 30&53&52&54&64&s^2&78&19&26&21&89& 1 \\ \hline
 31&49&65&44&66& 1&79&90&95&30&24& 0 \\ \hline
 32&46&67&66&47& 0&80&29&91&32&93& s \\ \hline
 33&53& 3&54& 5&s^2&81&95&33&90&28& s \\ \hline
 34&62&59&48&60& s&82&92&26&94&89& s \\ \hline
 35&65&68&63&42&s^2&83&33&90&34&32&s^2 \\ \hline
 36&52&53&64&54& s&84&95&95&90&24& s \\ \hline
 37&46&58&66&49& 0&85&35&91&36&93&s^2 \\ \hline
 38&55&68&49&42& s&86&37&35&89&94& 0 \\ \hline
 39&68&67&59&47& 1&87& 2&53& 4&54& 0 \\ \hline
 40&68&49&59&61& 1&88& 3&52& 5&64& 1 \\ \hline
 41&69&70&71& 8& 0&89&42&58&60&49&s^2 \\ \hline
 42&16&14& 7&15&s^2&90&59&42&61&44& 0 \\ \hline
 43&72&73&74&71& 1&91&67&65&57&66& 0 \\ \hline
 44&70&75& 8&76& 0&92&52&52&64&64& s \\ \hline
 45&17&75&16&76& 1&93&65&46&63&48&s^2 \\ \hline
 46&70&72& 8& 8& 0&94&67&55&57&59& 0 \\ \hline
 47&73&70&71&77& 0&95&62&55&48&59& s \\ \hline
 48&78&79&13&80& 1& & & & & &  \\ \hline
\end{array}\]

\section{Automata for generators of $\D_4^{\zeta}$}

\subsection{Automata for $\tau_1^{1,0},\tau_2^{1,0}$}
\label{3table}
\vspace{4mm}
\scriptsize
\[\begin{array}{|c||c|c|c|c||c|c|||c||c|c|c|c||c|c|}
\hline
\text{State}&0&1&2&3
&\tau_1^{1,0}\text{ label}&\tau_2^{1,0}\text{ label}
&\text{State}&0&1&2&3&\tau_1^{1,0}\text{ label}
&\tau_2^{1,0}\text{ label} \\ \hline
1&2&3&4&5&0&0&53&16&75&18&26&s&s^2 \\ \hline
2&6&7&8&9&0&0&54&7&8&9&47&0&0 \\ \hline
3&3&10&11&11&1&1&55&69&53&76&55&s^2&s \\ \hline
4&12&13&14&15&s^2&s&56&12&77&14&74&s^2&s \\ \hline
5&16&17&18&19&s&s^2&57&66&31&78&33&s&s^2 \\ \hline
6&6&20&8&8&0&0&58&21&16&59&33&0&0 \\ \hline
7&21&7&21&22&0&0&59&21&11&59&38&0&0 \\ \hline
8&23&24&25&26&1&1&60&10&21&30&21&0&0 \\ \hline
9&12&27&28&29&s^2&s&61&4&27&79&29&s^2&s \\ \hline
10&21&10&21&30&0&0&62&16&80&71&15&s&s^2 \\ \hline
11&16&31&32&33&s&s^2&63&20&4&81&34&0&0 \\ \hline
12&12&22&14&34&s^2&s&64&82&62&83&64&1&1 \\ \hline
13&12&35&15&21&s^2&s&65&23&84&52&38&1&1 \\ \hline
14&4&36&34&22&s^2&s&66&16&85&32&55&s&s^2 \\ \hline
15&4&13&37&15&s^2&s&67&82&24&86&26&1&1 \\ \hline
16&16&30&32&38&s&s^2&68&16&87&33&22&s&s^2 \\ \hline
17&23&39&40&15&1&1&69&12&88&14&43&s^2&s \\ \hline
18&11&41&42&43&s&s^2&70&89&49&19&46&1&1 \\ \hline
19&44&17&45&19&s&s^2&71&11&17&90&19&s&s^2 \\ \hline
20&21&20&21&46&0&0&72&12&91&61&33&s^2&s \\ \hline
21&21&21&21&21&0&0&73&20&11&92&38&0&0 \\ \hline
22&12&36&15&22&s^2&s&74&89&72&93&74&1&1 \\ \hline
23&23&46&25&47&1&1&75&23&94&40&34&1&1 \\ \hline
24&23&48&26&21&1&1&76&89&24&95&26&1&1 \\ \hline
25&8&49&47&46&1&1&77&12&96&15&30&s^2&s \\ \hline
26&8&24&50&26&1&1&78&82&49&29&46&1&1 \\ \hline
27&23&51&52&33&1&1&79&7&11&94&38&0&0 \\ \hline
28&4&53&54&55&s^2&s&80&12&92&28&47&s^2&s \\ \hline
29&56&27&57&29&s^2&s&81&23&72&40&74&1&1 \\ \hline
30&16&58&33&30&s&s^2&82&23&97&25&19&1&1 \\ \hline
31&16&59&33&21&s&s^2&83&56&13&98&15&s^2&s \\ \hline
32&11&58&38&30&s&s^2&84&16&41&71&43&s&s^2 \\ \hline
33&11&31&60&33&s&s^2&85&16&99&33&46&s&s^2 \\ \hline
34&7&4&4&34&0&0&86&56&36&43&22&s^2&s \\ \hline
35&21&4&35&34&0&0&87&21&44&59&55&0&0 \\ \hline
36&21&12&35&15&0&0&88&12&100&15&46&s^2&s \\ \hline
37&7&21&22&21&0&0&89&23&101&25&29&1&1 \\ \hline
38&10&11&11&38&0&0&90&10&4&84&34&0&0 \\ \hline
39&12&5&61&38&s^2&s&91&16&81&18&47&s&s^2 \\ \hline
40&8&62&63&64&1&1&92&23&62&52&64&1&1 \\ \hline
41&12&65&28&26&s^2&s&93&44&31&102&33&s&s^2 \\ \hline
42&10&8&5&47&0&0&94&12&53&61&55&s^2&s \\ \hline
43&66&41&67&43&s&s^2&95&44&58&55&30&s&s^2 \\ \hline
44&16&68&32&64&s&s^2&96&21&56&35&43&0&0 \\ \hline
45&69&13&70&15&s^2&s&97&23&103&26&22&1&1 \\ \hline
46&23&49&26&46&1&1&98&66&58&64&30&s&s^2 \\ \hline
47&20&8&8&47&0&0&99&21&66&59&64&0&0 \\ \hline
48&21&8&48&47&0&0&100&21&69&35&74&0&0 \\ \hline
49&21&23&48&26&0&0&101&23&104&26&30&1&1 \\ \hline
50&20&21&46&21&0&0&102&69&36&74&22&s^2&s \\ \hline
51&16&9&71&34&s&s^2&103&21&82&48&29&0&0 \\ \hline
52&8&72&73&74&1&1&104&21&89&48&19&0&0 \\ \hline
\end{array}\]

\subsection{Automata for $\tau_1^{s,0},\tau_2^{s^2,0}$}
\label{reflecttable}
\vspace{4mm}
\scriptsize
\[\begin{array}{|c||c|c|c|c||c|c|||c||c|c|c|c||c|c|}
\hline
\text{State}&0&1&2&3
&\tau_1^{s,0}\text{ label}&\tau_2^{s^2,0}\text{ label}
&\text{State}&0&1&2&3&\tau_1^{s,0}\text{ label}
&\tau_2^{s^2,0}\text{ label} \\
\hline
 1&2&3&4&5&0&0&53&52&59&68&43&1&1 \\ \hline
 2&6&7&8&9&0&0&54&19&15&64&43&0&0 \\ \hline
 3&3&10&8&4&1&1&55&8&69&70&71&s&s^2 \\ \hline
 4&11&12&13&14&s^2&s&56&11&72&73&24&s^2&s \\ \hline
 5&15&16&17&18&s&s^2&57&7&39&74&30&0&0 \\ \hline
 6&6&10&8&4&0&0&58&75&56&76&58&s^2&s \\ \hline
 7&19&20&19&21&0&0&59&15&64&43&19&s&s^2 \\ \hline
 8&15&22&23&24&s&s^2&60&11&77&24&21&s^2&s \\ \hline
 9&25&26&27&28&1&1&61&11&78&13&28&s^2&s \\ \hline
10&19&7&19&29&0&0&62&79&45&18&29&1&1 \\ \hline
11&11&29&13&30&s^2&s&63&20&19&21&19&0&0 \\ \hline
12&25&31&14&19&1&1&64&19&39&31&30&0&0 \\ \hline
13&8&32&33&21&s&s^2&65&25&80&38&36&1&1 \\ \hline
14&4&12&34&14&s^2&s&66&79&59&81&43&1&1 \\ \hline
15&15&35&23&36&s&s^2&67&15&82&43&35&s&s^2 \\ \hline
16&25&37&38&24&1&1&68&52&45&28&29&1&1 \\ \hline
17&39&26&40&28&1&1&69&15&83&84&43&s&s^2 \\ \hline
18&41&16&42&18&s&s^2&70&10&4&80&36&0&0 \\ \hline
19&19&19&19&19&0&0&71&85&69&86&71&s&s^2 \\ \hline
20&19&10&19&35&0&0&72&11&87&73&36&s^2&s \\ \hline
21&15&32&43&21&s&s^2&73&8&16&88&18&s&s^2 \\ \hline
22&11&44&24&19&s^2&s&74&25&89&38&90&1&1 \\ \hline
23&39&45&30&29&1&1&75&11&91&13&18&s^2&s \\ \hline
24&8&22&46&24&s&s^2&76&75&12&92&14&s^2&s \\ \hline
25&25&21&47&33&1&1&77&19&41&64&51&0&0 \\ \hline
26&25&48&27&14&1&1&78&25&93&14&35&1&1 \\ \hline
27&4&49&50&51&s^2&s&79&25&94&47&71&1&1 \\ \hline
28&52&26&53&28&1&1&80&11&49&55&51&s^2&s \\ \hline
29&25&45&14&29&1&1&81&41&32&51&21&s&s^2 \\ \hline
30&20&39&8&30&0&0&82&19&52&31&28&0&0 \\ \hline
31&19&4&44&36&0&0&83&15&95&84&33&s&s^2 \\ \hline
32&19&25&31&14&0&0&84&39&89&96&90&1&1 \\ \hline
33&10&8&4&33&0&0&85&15&97&23&90&s&s^2 \\ \hline
34&7&19&29&19&0&0&86&85&22&98&24&s&s^2 \\ \hline
35&11&54&24&35&s^2&s&87&11&56&73&58&s^2&s \\ \hline
36&7&4&39&36&0&0&88&10&39&87&30&0&0 \\ \hline
37&11&5&55&33&s^2&s&89&11&99&55&43&s^2&s \\ \hline
38&4&56&57&58&s^2&s&90&79&89&100&90&1&1 \\ \hline
39&25&59&47&43&1&1&91&25&101&14&21&1&1 \\ \hline
40&20&8&5&33&0&0&92&75&54&58&35&s^2&s \\ \hline
41&15&60&23&58&s&s^2&93&19&61&44&90&0&0 \\ \hline
42&61&12&62&14&s^2&s&94&15&102&43&29&s&s^2 \\ \hline
43&39&59&63&43&1&1&95&15&69&84&71&s&s^2 \\ \hline
44&19&8&64&33&0&0&96&20&4&95&36&0&0 \\ \hline
45&19&11&44&24&0&0&97&11&103&24&29&s^2&s \\ \hline
46&10&19&35&19&0&0&98&85&32&71&21&s&s^2 \\ \hline
47&4&54&36&35&s^2&s&99&15&74&17&30&s&s^2 \\ \hline
48&25&9&27&30&1&1&100&41&22&104&24&s&s^2 \\ \hline
49&15&65&17&14&s&s^2&101&19&75&44&58&0&0 \\ \hline
50&7&8&9&33&0&0&102&19&79&31&18&0&0 \\ \hline
51&61&49&66&51&s^2&s&103&19&85&64&71&0&0 \\ \hline
52&25&67&47&51&1&1&104&61&54&90&35&s^2&s \\ \hline
\end{array}\]

\subsection{Automata for $\tau_2^{s,0},\tau_1^{s^2,0}$}
\label{rottable}
\vspace{4mm}
\scriptsize
\[\begin{array}{|c||c|c|c|c||c|c|||c||c|c|c|c||c|c|}
\hline
\text{State}&0&1&2&3&\tau_2^{s,0}\text{ label}
&\tau_1^{s^2,0}\text{ label}&\text{State}&0&1&2&3
&\tau_2^{s,0}\text{ label}&\tau_1^{s^2,0}\text{ label} \\
\hline
 1&2&3&4&5&0&0&53&20&46&32&48&1&1 \\ \hline
 2&6&7&8&9&0&0&54&13&16&65&38&s&s^2 \\ \hline
 3&3&10&11&12&1&1&55&16&68&18&44&s^2&s \\ \hline
 4&13&14&13&15&s&s^2&56&7&8&9&32&1&1 \\ \hline
 5&16&17&18&19&s^2&s&57&7&67&25&13&1&1 \\ \hline
 6&6&20&8&21&0&0&58&4&13&48&13&s&s^2 \\ \hline
 7&7&15&22&23&1&1&59&13&10&70&77&s&s^2 \\ \hline
 8&7&24&9&25&1&1&60&29&11&78&28&0&0 \\ \hline
 9&26&11&27&28&0&0&61&26&50&79&19&0&0 \\ \hline
10&29&30&31&32&0&0&62&14&55&80&23&s&s^2 \\ \hline
11&16&33&34&25&s^2&s&63&16&81&38&48&s^2&s \\ \hline
12&35&36&37&38&1&1&64&82&83&84&64&s^2&s \\ \hline
13&13&13&13&13&s&s^2&65&13&45&67&47&s&s^2 \\ \hline
14&13&39&13&40&s&s^2&66&14&13&15&13&s&s^2 \\ \hline
15&16&41&38&15&s^2&s&67&13&26&70&42&s&s^2 \\ \hline
16&16&40&18&42&s^2&s&68&29&70&44&13&0&0 \\ \hline
17&29&43&44&15&0&0&69&39&13&40&13&s&s^2 \\ \hline
18&45&46&47&48&1&1&70&13&55&65&23&s&s^2 \\ \hline
19&49&50&51&19&0&0&71&7&85&25&15&1&1 \\ \hline
20&7&52&22&28&1&1&72&5&86&87&72&s^2&s \\ \hline
21&20&36&53&38&1&1&73&29&88&74&42&0&0 \\ \hline
22&26&54&42&40&0&0&74&55&86&89&72&s^2&s \\ \hline
23&39&55&26&23&s&s^2&75&49&54&19&40&0&0 \\ \hline
24&7&56&9&47&1&1&76&13&20&67&32&s&s^2 \\ \hline
25&26&57&58&25&0&0&77&35&90&91&77&1&1 \\ \hline
26&29&57&31&25&0&0&78&55&83&92&64&s^2&s \\ \hline
27&4&55&56&23&s&s^2&79&4&45&93&47&s&s^2 \\ \hline
28&10&11&12&28&0&0&80&16&86&34&72&s^2&s \\ \hline
29&29&48&31&47&0&0&81&13&35&67&72&s&s^2 \\ \hline
30&7&59&25&40&1&1&82&16&94&18&77&s^2&s \\ \hline
31&55&41&23&15&s^2&s&83&16&95&96&38&s^2&s \\ \hline
32&20&8&21&32&1&1&84&82&68&97&44&s^2&s \\ \hline
33&7&60&61&42&1&1&85&13&49&70&19&s&s^2 \\ \hline
34&45&8&62&32&1&1&86&7&98&61&44&1&1 \\ \hline
35&7&63&22&64&1&1&87&10&57&99&25&0&0 \\ \hline
36&16&65&38&13&s^2&s&88&29&50&74&19&0&0 \\ \hline
37&5&41&28&15&s^2&s&89&39&45&88&47&s&s^2 \\ \hline
38&45&36&66&38&1&1&90&29&100&78&38&0&0 \\ \hline
39&13&4&13&48&s&s^2&91&5&68&101&44&s^2&s \\ \hline
40&29&54&44&40&0&0&92&39&26&60&42&s&s^2 \\ \hline
41&13&7&67&25&s&s^2&93&7&90&61&77&1&1 \\ \hline
42&4&26&45&42&s&s^2&94&29&102&44&48&0&0 \\ \hline
43&13&5&65&28&s&s^2&95&16&103&96&23&s^2&s \\ \hline
44&55&68&69&44&s^2&s&96&45&90&104&77&1&1 \\ \hline
45&7&36&22&38&1&1&97&82&41&64&15&s^2&s \\ \hline
46&13&29&70&44&s&s^2&98&29&80&78&23&0&0 \\ \hline
47&14&45&55&47&s&s^2&99&35&46&72&48&1&1 \\ \hline
48&7&46&25&48&1&1&100&16&93&34&47&s^2&s \\ \hline
49&29&71&31&72&0&0&101&10&54&77&40&0&0 \\ \hline
50&29&73&74&44&0&0&102&13&82&65&64&s&s^2 \\ \hline
51&49&57&75&25&0&0&103&16&83&96&64&s^2&s \\ \hline
52&16&76&38&40&s^2&s&104&14&26&103&42&s&s^2 \\ \hline
\end{array}\]

\end{document}